\documentclass[12pt, reqno]{amsart}
\usepackage{fullpage,url,amssymb,mathrsfs}
\usepackage{color}
\usepackage[alphabetic,lite]{amsrefs}

\newcommand{\bvchange}[1]{{
			#1}}

\newtheorem{lemma}{Lemma}[section]
\newtheorem{theorem}[lemma]{Theorem}
\newtheorem{prop}[lemma]{Proposition}
\newtheorem{cor}[lemma]{Corollary}
\newtheorem{conj}[lemma]{Conjecture}

\newtheorem{claim*}{Claim}

\newtheorem{remark}[lemma]{Remark}
\newtheorem{thm}[lemma]{Theorem}

\newtheorem*{theorem*}{Theorem}

\newcommand{\C}{{\mathbb C}}
\newcommand{\F}{{\mathbb F}}
\newcommand{\Q}{{\mathbb Q}}
\newcommand{\R}{{\mathbb R}}
\newcommand{\Z}{{\mathbb Z}}

\newcommand{\BBl}{{\mathbb B}_{\ell, \infty}}

\newcommand{\WW}{{\mathbb W}}
\newcommand{\EE}{{\mathbb E}}

\newcommand{\Qbar}{{\overline{\Q}}}

\newcommand{\Fbar}{{\overline{\F}}}

\newcommand{\Dtilde}{{\widetilde{D}}}

\newcommand{\Rtilde}{{\widetilde{R}}}

\newcommand{\mm}{{\mathfrak m}}

\newcommand{\calE}{{\mathcal E}}

\newcommand{\OO}{{\mathcal O}}

\newcommand{\frakb}{\mathfrak{b}}

\newcommand{\frakp}{\mathfrak{p}}

\newcommand{\frakA}{\mathfrak{A}}

\newcommand{\frakI}{\mathfrak{I}}

\DeclareMathOperator{\Tr}{Tr}

\DeclareMathOperator{\im}{Im}

\DeclareMathOperator{\End}{End}

\DeclareMathOperator{\Hom}{Hom}

\DeclareMathOperator{\Aut}{Aut}

\DeclareMathOperator{\Norm}{N}

\DeclareMathOperator{\Pic}{Pic}
\DeclareMathOperator{\Jac}{Jac}

\DeclareMathOperator{\Mat}{M}
\DeclareMathOperator{\Disc}{D}
\DeclareMathOperator{\SL}{SL}
\DeclareMathOperator{\GL}{GL}

\DeclareMathOperator{\CM}{CM}

\DeclareMathOperator{\RO}{RO}

\newcommand{\into}{\hookrightarrow}

\numberwithin{equation}{section}
\numberwithin{table}{section}

\title{An arithmetic intersection formula for denominators of Igusa class polynomials}

\author{Kristin Lauter}
\author{Bianca Viray}
\thanks{The second author was partially supported by Microsoft Research, a Ford dissertation year fellowship, NSF grant DMS-1002933, and ICERM}

\address{Microsoft Research, 1 Microsoft Way, Redmond, WA 98062, USA}
\email{klauter@microsoft.com}
\urladdr{http://research.microsoft.com/en-us/people/klauter/default.aspx}

\address{University of Washington, Department of Mathematics, Box 354350, Seattle, WA 98195, USA}
\email{bviray@math.washington.edu}
\urladdr{http://math.washington.edu/\~{}bviray}

\date{}

\begin{document}

	\begin{abstract}
		In this paper we prove an explicit formula for the arithmetic intersection number $(\CM(K).\textup{G}_1)_{\ell}$ on the Siegel moduli space of abelian surfaces, generalizing the work of Bruinier-Yang and Yang.
These intersection numbers allow one to compute the denominators of Igusa class polynomials, which has important applications to the construction of genus $2$ curves for use in cryptography.
		
		Bruinier and Yang conjectured a formula for intersection numbers on an arithmetic Hilbert modular surface, and as a consequence obtained a conjectural formula for the intersection number $(\CM(K).\textup{G}_1)_{\ell}$ under strong assumptions on the ramification of the primitive quartic CM field $K$.  Yang later proved this conjecture assuming that $\OO_K$ is freely generated by one element over the ring of integers of the real quadratic subfield.  In this paper, we prove a formula for $(\CM(K).\textup{G}_1)_{\ell}$ for more general primitive quartic CM fields, and  we use a different method of proof than Yang.  We prove a tight bound on this intersection number which holds for {\it all} primitive quartic CM fields.  As a consequence, we obtain a formula for a multiple of the denominators of the Igusa class polynomials for an arbitrary primitive quartic CM field.  Our proof entails studying the Embedding Problem posed by Goren and Lauter and counting solutions using our previous article that generalized work of Gross-Zagier and Dorman to arbitrary discriminants.  
	\end{abstract}

	\maketitle
			

\section{Introduction}\label{sec:intro}

For a prime number $\ell$, the $\ell$-part of the arithmetic intersection number $(\CM(K).\textup{G}_1)$ counts, with multiplicity,
the number of isomorphism classes of abelian surfaces with CM by a primitive quartic CM field $K$ that reduce modulo $\ell$ to 
a product of two elliptic curves with the product polarization.  These intersection numbers have been studied in detail by Bruinier-Yang~\cite{BY}, Yang~\cite{Yang1, Yang2}, and Goren-Lauter~\cite{GL, GL10}.
In this paper, we give an exact formula for this $\ell$-part, denoted $(\CM(K).\textup{G}_1)_{\ell}$, under mild assumptions on $K$, and a tight bound on $(\CM(K).\textup{G}_1)_{\ell}$ for all primitive quartic CM fields $K$.

	The computation of $(\CM(K).\textup{G}_1)_{\ell}$ has applications to the computation of the Igusa class polynomials of $K$.  Igusa class polynomials are polynomials over $\Q$ which are the genus $2$ analogue of Hilbert class polynomials; namely, the roots of the Igusa class polynomials of $K$ determine genus $2$ curves whose Jacobians have complex multiplication by $K$.  However, in contrast to the genus $1$ case, the coefficients of Igusa class polynomials are \emph{not} integral and the presence of denominators makes the computation of these polynomials more difficult.  Indeed, all known algorithms to compute Igusa class polynomials require as input some bound on the denominators of the coefficients of the Igusa class polynomials.  In addition, the sharpness of the bound directly affects the efficiency of the algorithms.  
	The arithmetic intersection number $\CM(K).\textup{G}_1$ gives a method of studying these denominators.  In fact, up to cancellation from the numerators, the $\ell$-valuation of the denominators of Igusa class polynomials is exactly a (known) multiple of $(\CM(K).\textup{G}_1)_{\ell}$.

	Often, explicit formulas for the arithmetic intersection of CM-cycles with other cycles, such as the Humbert surface, are proved under severe restrictions on the ramification in the CM field $K$ (e.g. \cite{GZ-SingularModuli}).  Indeed, Yang proved an explicit formula for $(\CM(K).\textup{G}_1)_{\ell}$ under the assumption that the discriminant of $K$ is of the form $D^2\Dtilde$ where $D$ and $\Dtilde$ are primes congruent to $1 \pmod{4}$~\cite{Yang1, Yang2}, and that $\OO_K$ is freely generated over the ring of integers of the real quadratic subfield by one element of a certain form.  

This explicit formula was originally conjectured, with the assumption on the ramification but without the assumption on $\OO_K$, in earlier work of Bruinier and Yang~\cite{BY}.  
In recent work, the present authors with Grundman, Johnson-Leung, Salerno, and Wittenborn~\cite{WIN} showed that the conjecture of Bruinier and Yang does not hold (as stated) if the assumptions on the ramification are relaxed.  This gives evidence that, in the general case, the formula \emph{must} be more complicated.

The main result of this paper is an explicitly computable formula for the intersection number $(\CM(K).\textup{G}_1)_{\ell}$,  under the same assumption on $\OO_K$, for all $\ell$ outside a small finite set, and a tight upper bound for $(\CM(K).\textup{G}_1)_{\ell}$ in general (\S\ref{sec:MainResult}).  The dramatically weaker assumptions lead to a formula that is more complicated than that of Bruinier and Yang; however, in many cases it simplifies to a formula that is strikingly similar.  We give an example of this in~\S\ref{sec:WIN2}.	
	As a result of our formula and upper bound, we obtain a formula for a multiple of the denominators of the Igusa class polynomials for every primitive quartic CM field $K$.  We explain this further in~\S\ref{sec:Igusa}.
	
	\smallskip
	
	\noindent\textit{Remark. }
		The arithmetic intersection number $\CM(K).\textup{G}_1$ was also studied by Howard and Yang~\cite{HY}. They prove, under very mild assumptions, that the values $(\CM(K).\textup{G}_1)_{\ell}$ agree with Fourier coefficients of certain Eisenstein series; however, their work does not give an explicit formula for $(\CM(K).\textup{G}_1)_{\ell}$.

	\subsection{Overview of the tools}
		The first part of our proof takes its inspiration from work of Goren and the first author~\cite{GL, GL10} which gave a bound on the denominators appearing in the Igusa class polynomials,  first bounding the primes that can appear~\cite{GL}, and then bounding the powers~\cite{GL10}.  Their proof studied necessary conditions for the existence of a solution to the \textit{embedding problem}: the problem of determining whether there is an embedding $\OO_K\hookrightarrow \End_{\Fbar_{\ell}}(E_1\times E_2)$ such that complex conjugation agrees with the Rosati involution associated to the product polarization.  
		
		In this paper, we determine conditions that are \emph{equivalent} to the existence of a solution to the embedding problem and use these equivalent conditions to count the number of solutions to the embedding problem.  (Yang's proof~\cite{Yang1, Yang2} also began with a treatment of the embedding problem; however, our formulation of it is different and our methods diverge from Yang's after this step.)  First, we show that a solution to the embedding problem gives rise to a supersingular elliptic curve $E$ and endomorphisms $x,u\in\End(E)$ with fixed degree and trace. \bvchange{We  explain this  under assumption $(\dagger)$ in~\S\ref{sec:ProofOfMainTheorem}, and explain the modifications needed to lift the assumption $(\dagger)$ in~\S\ref{subsec:ProofOfUpperBound}.}
		
		Next, we count these pairs of endomorphisms using results from our earlier paper~\cite{LV-SingularModuli} that generalizes work of Gross and Zagier~\cite{GZ-SingularModuli}.  These results show that the number of pairs $(x,u)$ is equal to a weighted sum of the number of integral ideals in a quadratic imaginary order with a certain norm.  This is explained further in~\S\ref{sec:LVSingularModuli}.
		
		To go from pairs of endomorphisms $(x,u)$ to a solution of the embedding problem, we study isogenies $y,b$ of a fixed degree from an auxiliary elliptic curve $E'$ to $E$ such that $yb^{\vee} = u$ and such that $x, y,$ and $b$ satisfy an additional relationship depending on $K$.  Using Deuring's correspondence, we translate this to a problem of counting certain ideals in $\Mat_2(\Z_p)$.  We solve this counting problem in~\S\ref{sec:IdealsInM2Zp}.

\section*{Acknowledgements}
	 The first author thanks Tonghai Yang for explanations of his work~\cite{Yang1, Yang2} and the second author thanks Ben Howard, Bjorn Poonen, and Joseph H. Silverman for helpful conversations.  In addition, we thank the referee for his or her comments and careful reading of the manuscript.  


\section{A formula for $\CM(K)\cdot\textup{G}_1$}\label{sec:ThmStatement}

	\subsection*{Notation}
		We write $F$ for a real quadratic field, and $D$ for the discriminant of the ring of integers $\OO_F$.  Let $K$ denote a totally imaginary extension of $F$ that does not contain an imaginary quadratic field; $K$ is a {\it primitive quartic CM field}.  We say that an abelian surface $A$ has {\it CM by }$K$ if there is an embedding of the ring of integers $\OO_K$ into the endomorphism ring $\End(A)$.  Let $\CM(K)$ denote the moduli stack whose $S$-points are
		\[
			\left\{
			\begin{array}{ll}
				(A,\iota,\lambda):&A/S
				\textup{ is an abelian surface with principal
				polarization }\lambda,\\
			 	&\iota\colon\OO_K\hookrightarrow\End_S(A), \textup{ s.t. }
				\iota(\overline{\gamma}) = \iota(\gamma)^{\vee}
			\end{array}
			\right\}/\sim
		\]  
		where $\phi\mapsto\phi^{\vee}$ denotes the Rosati involution and $(A,\iota,\lambda)\sim(A',\iota',\lambda')$ if there is an isomorphism of principally polarized abelian surfaces between $(A,\lambda)$ and $(A',\lambda')$ that conjugates $\iota$ to $\iota'$.  There is a finite to one map from $\CM(K)$ to $M$, the Siegel moduli space of principally polarized abelian surfaces, obtained by sending $(A, \iota, \lambda)$ to $(A, \lambda)$.
		
		Let $\eta$ denote a fixed element of $\OO_K$ that generates $K/F$.  Often, we will assume that $\OO_K$ is freely generated over $\OO_F$ and that $\eta$ is a generator, i.e.,
		\begin{equation}\label{eq:assumption}
			\OO_K = \OO_F[\eta]\tag{$\dagger$}
		\end{equation}
		
		\bvchange{In order to state the main result, we need to introduce some constants.  The origin of these constants will become clear in~\S\ref{sec:ProofOfMainTheorem}; for now, it is enough to note that these values are easily computed once a choice of $\eta$ is fixed.}
		
		We write $\Dtilde := \Norm_{F/\Q}\left(\Disc_{K/F}(\eta)\right)$ where $\Disc_{K/F}(\eta)$ denotes the relative discriminant of $\OO_F[\eta]$ and we let $\alpha_0,\alpha_1,\beta_0,\beta_1\in\Z$ be such that $\Tr_{K/F}(\eta) = \alpha_0 + \alpha_1\omega$ and $\Norm_{K/F}(\eta) = \beta_0 + \beta_1\omega$, where $\omega = \frac12(D + \sqrt{D})$.  We define 
		\begin{equation}
			c_K := \Norm_{F/\Q}(\Tr_{K/F}(\eta)) - 2\Tr_{F/\Q}(\Norm_{K/F}(\eta)) = \alpha_0^2 + \alpha_0\alpha_1D + \frac14\alpha_1^2({D^2 - D}) - 4\beta_0 - 2\beta_1D.\label{eq:cK}
		\end{equation}
	
		For any positive integer $\delta$ such that $D - 4\delta$ is a square, we define $t_u(\delta) := \alpha_1\delta$ and define $t_x(\delta), t_w(\delta)\in\Z$ to satisfy
		\begin{equation}
			t_x(\delta) + t_w(\delta) = 2\alpha_0 + D\alpha_1, \quad
			t_w(\delta) - t_x(\delta) = \alpha_1\sqrt{D - 4\delta}, 
			\quad\textup{ where }\sqrt{D - 4\delta}>0.\label{eq:traces}
		\end{equation}
		\bvchange{Let $n$ be any integer $n$ such that $2D\mid(n + c_K\delta)$; note that since $n$ depends on a choice of $\delta$, so too will any constant that depends on $n$.  For simplicity, we omit this dependence on $\delta$ in the notation.  We define} 
		\begin{equation}
			n_u(n) := -\delta\cdot\frac{n + c_K\delta}{2D}, \quad
			t_{xu^{\vee}}(n) := \beta_1\delta + \sqrt{D - 4\delta}\frac{n + c_K\delta}{2D},\label{eq:normu}
		\end{equation}
		and let $n_x(n), n_w(n)\in\Z$ be such that
		\begin{equation}
			n_x(n) + n_w(n) = 2\beta_0 + D\beta_1 - 2n_u(n)/\delta,\quad
			n_w(n) - n_x(n) = \beta_1\sqrt{D - 4\delta}.\label{eq:norms}
		\end{equation}
		We also define $d_*(n) := t_*(\delta)^2 - 4n_*(n)$ for $*\in\{x,u,w\}$.  For any positive integer $f_u$, set
		\begin{equation}
			t(n, f_u) = \frac1{2f_u^2}(d_x(n)d_u(n) - f_u(t_x(\delta)t_u(\delta) - 2t_{xu^\vee}(n))).\label{eq:trxu}
		\end{equation}
		%
		
		Throughout we work with a fixed prime $\ell$; we write $\BBl$ for the rational quaternion algebra ramified only at $\ell$ and $\infty$.  For any $\gamma\in\BBl$, we let $\gamma^{\vee}$ denote the image of $\gamma$ under the natural involution and define $\Tr(\gamma) := \gamma + \gamma^{\vee}, \Norm(\gamma) := \gamma\gamma^{\vee}.$
		
	\subsection{The main result}\label{sec:MainResult}

		\begin{thm}\label{thm:main}\label{thm:MainResult}
			Assume $(\dagger)$.  If $\ell\nmid \delta$ for any positive integer $\delta$ such that $D - 4\delta$ is a square, then
		\[
			\frac{(\CM(K).G_1)_{\ell}}{\log \ell} =
				\sum_{\substack{\delta \in \Z_{>0}\\ D - 4\delta=\square}}
				C_{\delta}
				\sum_{\substack{n\in\Z\\
					\frac{\delta^2\Dtilde - n^2}{4D} \in\ell\Z_{>0}\\
					2D\mid(n + c_K\delta) }}\mu_{\ell}(n)
				\sum_{f_u}\frakI(n, f_u)\cdot\mathscr{J}
				\left(d_u(n)f_u^{-2}, d_x(n), t(n, f_u)\right).
		\]
		Otherwise,
		\[
			\frac{(\CM(K).G_1)_{\ell}}{\log \ell} \leq
				2\sum_{\substack{\delta \in \Z_{>0}\\ D - 4\delta =\square}}
				C_{\delta}\sum_{\substack{n\in\Z\\
				\frac{\delta^2\Dtilde - n^2}{4D} \in\ell\Z_{>0}\\
				2D\mid(n + c_K\delta) }}\mu_{\ell}(n)
				\sum_{f_u}\frakI(n, f_u)\cdot\mathscr{J}
								\left(d_u(n)f_u^{-2}, d_x(n), t(n, f_u)\right).
		\]
		Here $C_{\delta} = 2$ if $4\delta = D$ and otherwise $C_{\delta} = 1$, and $\mu_{\ell}(n) = v_{\ell}\left(\frac{\delta^2\Dtilde - n^2}{4D}\right)$ if $\ell$ divides both $d_u(n)$ and $d_x(n)$ and $\mu_{\ell}(n) = \frac12(v_{\ell}(\frac{\delta^2\Dtilde - n^2}{4D}) + 1)$ otherwise.  
		The sum $\sum_{f_u}$ ranges over positive integers $f_u$ such that $d_u(n)/f_u^2$ is the discriminant of a quadratic imaginary order that is maximal at $\ell$.  The quantity $\mathscr{J}\left(d_1, d_2, t\right)$ is a sum, over
		isomorphism classes of supersingular elliptic curves modulo $\ell$, of a number of pairs of embeddings, precisely $\mathscr{J}\left(d_1,d_2, t\right)$ equals
		\[
			\sum_{E/\Fbar_{\ell}}\#\left\{
			\begin{array}{ll}
				(i_1,i_2), i_j\colon\Z\left[\frac{d_j + \sqrt{d_j}}2\right]\hookrightarrow 
				\End(E) : &
				\Tr(i_1(d_1 + \sqrt{d_1})i_2(d_2 - \sqrt{d_2})) = 4t,\\
				& i_1(\Q(\sqrt{d_1}))\cap\End(E) = \Z\left[\frac{d_1 + \sqrt{d_1}}2\right]
			\end{array}
			\right\}/\End(E)^{\times}.
		\]
		Lastly,
		\[
			\frakI(n, f_u) := \prod_{p|\delta, p\neq \ell}
			\left(
			\sum_{\substack{j = 0\\j\equiv v_p(\delta)\bmod2}}^{v_p(\delta)}
			 \frakI^{(p)}_{j - r_p}(t_w(n), n_w(n))\right) ,
		\]
		where $r_p := \max\left(v_p(\delta) - \min(v_p(f_u), v_p(\frac{d_u(n) - t_uf_u}{2f_u})), 0\right)$ and
		\[
			\frakI_{C}^{(p)}(a_1, a_0) = 
			\begin{cases}
				\#\{\widetilde{t} \bmod p^{C} : 
					\widetilde{t}^2 - a_1\widetilde{t} + a_0
					\equiv 0 \pmod{p^{C}}\}
				& \textup{ if }  C\geq 0,\\
				0 & \textup{ if } C< 0.
			\end{cases}
		\]
	\end{thm}
	\begin{remark}
	The quantity $\mathscr{J}\left(d_1, d_2, t\right)$ can be computed, for any given $K$, via an algorithm presented in~\cite{WIN}. Additionally, Theorem~\ref{thm:LVSingularModuli} below will give a formula for $\mathscr{J}\left(d_1, d_2, t\right)$ in most cases, and an upper bound for $\mathscr{J}\left(d_1, d_2, t\right)$ in the remaining cases.  \bvchange{  Furthermore} Conjecture~\ref{conj:localfactors} and Remark~\ref{rem:localfactors} give an even simpler expression for $\mathscr{J}\left(d_1, d_2, t\right)$ as a product of local factors which holds under some additional assumptions.
	\end{remark}	

	If $\OO_K$ is not freely generated over $\OO_F$, then the same methods give an upper bound for the arithmetic intersection number.
	\begin{thm}\label{thm:UpperBound}
		For every $\eta\in\OO_K$ such that $[\OO_K : \OO_F[\eta]]$ is relatively prime to $\ell$ and all primes $p \leq D/4$, we have an upper bound:
		\[
			\frac{(\CM(K).G_1)_{\ell}}{\log \ell} \leq
				2\sum_{\substack{\delta \in \Z_{>0}\\ D - 4\delta=\square}}
				C_{\delta}
				\sum_{\substack{n\in\Z\\
					\frac{\delta^2\Dtilde - n^2}{4D} \in\ell\Z_{>0}\\
					2D\mid(n + c_K\delta) }}\mu_{\ell}(n)
				\sum_{f_u}\frakI(n, f_u)\cdot\mathscr{J}
				\left(d_u(n)f_u^{-2}, d_x(n), t(n, f_u)\right),
		\]
		with the notation as in Theorem~\ref{thm:main}.
	\end{thm}
	
	The quantity $\mathscr{J}\left(d_1, d_2, t\right)$ is related to the $\ell$ valuation of $J(d_1, d_2) = \prod_{[\tau_1],[\tau_2]}(j(\tau_1) - j(\tau_2))$.  It was considered first in 1985 by Gross and Zagier in the case that $d_1$ and $d_2$ are discriminants of imaginary quadratic fields and that $d_1$ and $d_2$ are relatively prime~\cite{GZ-SingularModuli}.  The present authors recently generalized much of~\cite{GZ-SingularModuli} to arbitrary discriminants~\cite{LV-SingularModuli}.  As $d_1 = d_u(n)f_u^{-2}$ and $d_2 = d_x(n)$ are not necessarily relatively prime nor necessarily discriminants of \emph{maximal} orders, this generalization is needed to compute $\mathscr{J}\left(d_1, d_2, t\right)$ and thus to give a formula for $(\CM(K).G_1)_{\ell}.$  Using results from~\cite{LV-SingularModuli}, we obtain the following theorem.
	
	\begin{thm}\label{thm:LVSingularModuli}
		Fix $n, f_u\in\Z$ as above, set $d_x := d_x(n), d_u := d_u(n), t := t(n, f_u)$, and write $\OO_u$ for the quadratic imaginary order of discriminant $d_u/f_u^2$.  If the Hilbert symbol 
		\[
		(d_u, D(n^2 - \delta^2\Dtilde))_p = (d_u, (d_uf_u^{-2}d_x - 2t)^2 - d_u{f_u}^{-2}d_x)_p
		\] is equal to $-1$ for some prime $p\ne \ell$, then $\mathscr{J}\left(d_uf_u^{-2}, d_x, t\right) = 0$.  If $\frac{\delta^2\Dtilde - n^2}{4Df_u^2}$ is coprime to the conductor of $\OO_u$, then $\mathscr{J}\left(d_uf_u^{-2}, d_x, t\right)$ equals
		\begin{equation}\label{eq:BoundForscrJ}
			2^{\#\{p\; :\; v_p(t) \geq v_p(d_uf_u^{-2}) > 0, p\nmid 2\ell\}}
			\cdot\tilde\rho_{d_uf_u^{-2}}^{(2)}(t, d_x)\cdot
			\#\left\{\frakb\subseteq\OO_u
			: \Norm(\frakb) = \frac{\delta^2\Dtilde - n^2}{4D\ell f_u^2},
			\frakb \textup{ invertible}\right\},
		\end{equation}
		where 
		\[
		\widetilde{\rho}^{(2)}_d(s_0,s_1)  := 
			\left\{
			\begin{array}{ll}
				2 & \textup{if } d \equiv 12 \bmod{16}, 
					s_0\equiv s_1\bmod2\\
				 & \textup{or if } 8\mid d, v_2(s_0) \ge v_2(d) - 2\\
				1& \textup{otherwise}
			\end{array}
			\right\}
			\cdot
			\left\{
			\begin{array}{ll}
				2 & \textup{if } 32\mid d, 4\mid (s_0 - 2s_1)\\
				1& \textup{otherwise}
			\end{array}
			\right\},
		\]
		\bvchange{if $\ell \neq 2$ and $\widetilde{\rho}^{(2)}_d(s_0,s_1) = 1$ if $\ell =2$.}
		Furthermore, in all cases, $\mathscr{J}\left(d_uf_u^{-2}, d_x, t\right)$ is bounded above by~\eqref{eq:BoundForscrJ} and there is an algorithm to compute $\mathscr{J}\left(d_uf_u^{-2}, d_x, t\right)$.
	\end{thm}
	\begin{remark}
		If $\frac{\delta^2\Dtilde - n^2}{4Df_u^2}$ is coprime to the conductor of $\OO_u$, then the quantity 
		\[
			2^{\#\{p\; :\; v_p(t) \geq v_p(d_uf_u^{-2}) > 0, p\nmid 2\ell\}}
			\cdot\tilde\rho_{d_uf_u^{-2}}^{(2)}(t, d_x)
		\]
		simplifies to
		\[
			2^{\#\{p\; :\;p|\textup{gcd}(Nf_u^{-2}, d_uf_u^{-2}), p\nmid \ell\}}
		\]
		where $N = \frac{\delta^2\Dtilde - n^2}{4D}$.
	\end{remark}

	In the case that $\frac{\delta^2\Dtilde - n^2}{4D}$ is coprime to the conductor, we can also express $\mathscr{J}\left(d_uf_u^{-2}, d_x, t\right)$ as a product of local factors.  This expression leads us to the following conjecture.
	\begin{conj} \label{conj:localfactors}
		Let $d_1$ and $d_2$ be discriminants of quadratic imaginary orders and fix an integer $t$.  Assume that conductor of $d_1$, the conductor of $d_2$, and $m := \frac14(d_1d_2 - (d_1d_2 - 2t)^2)$ have no simultaneous common factor.
		 Then
		\[
		\mathscr{J}\left(d_1, d_2, t\right) = \prod_{p|m, p\ne\ell}
					\begin{cases}
						1 + v_p(m) & \left(\frac{d_{(p)}}{p}\right) = 1, p\nmid f_1,\\
						2 & \left(\frac{d_{(p)}}{p}\right) = 1, p|f_1, \textup{ or}\\
						& p|d_{(p)}, (d_{(p)}, -m)_p = 1, p\nmid f_1\\
						1 & \left(\frac{d_{(p)}}{p}\right) = -1, p\nmid f_1, v_p(m)\textup{ even} \textup{ or}\\
						& p|d_{(p)}, (d_{(p)}, -m)_p = 1, p |f_1 , v_p(m) = 2\\
						0 & \textup{otherwise},
					\end{cases}
		\]
		where $d_{(p)}\in\{d_1,d_2\}$ is such that the quadratic imaginary order of discriminant $d_{(p)}$ is maximal at $p$, $f_1$ denotes the conductor of $d_1$, \bvchange{and $\left(\frac{d_{(p)}}{p}\right)$ denotes the Kronecker symbol of $d_{(p)}$ at $p$}.
	\end{conj}
	\begin{remark} \label{rem:localfactors}
		This conjecture holds when $f_1$ and $m$ are coprime; in that case it follows from Theorem~\ref{thm:LVSingularModuli}.
	\end{remark}

	Together, Theorems~\ref{thm:MainResult} and~\ref{thm:LVSingularModuli} give a sharp bound on $(\CM(K).G_1)_{\ell}$ for all primes $\ell$, and a sharp bound on the primes $\ell$ such that $(\CM(K).G_1)_{\ell}\neq 0$.  The following Corollary gives a characterization of these primes.
	
	\begin{cor}\label{cor:Divisibility}
		Assume $(\dagger)$ and that $(\CM(K).G_1)_{\ell}\neq 0$.  Then there exists a $\delta\in\Z_{>0}$ and $n\in\Z$ such that $D - 4\delta$ is a square, $n \equiv -c_K\delta\pmod{2D}$,
		\[
			N := \frac{\delta^2\Dtilde- n^2}{4D} \in \ell\Z_{>0}, \quad 
			\textup{ and }\quad
			(d_u(n), -N)_p = \begin{cases} 1& \textup{if }p\neq \ell,\\
			-1& p = \ell.\end{cases}
		\]
	\end{cor}
	\begin{remark}
		One obtains the same corollary even when $\OO_K$ is not generated over $\OO_F$ by one element, by replacing $(\dagger)$ with the assumption that $\OO_F[\eta]$ is maximal at $\ell$ and all prime $p\leq D/4$.  Note that different choices of $\eta\in\OO_K$ result in different values of $\Dtilde = \Norm_{F/\Q}(\Disc_{K/F}(\eta))$ and each choice results in a valid upper bound.
	\end{remark}
	\begin{proof}
		By Theorem~\ref{thm:MainResult}, $(\CM(K).\textup{G}_1)_{\ell}$ is always bounded above by a sum over $\delta\in\Z_{>0}$ such that $D - 4\delta$ is a square and a sum over $n\in\Z$ such that $2D|(n + c_K\delta)$ and such that $N := \frac{\delta^2\Dtilde - n^2}{4D}$ is a positive integer divisible by $\ell$.  Thus, it remains to show that if $(\CM(K).G_1)_{\ell}\neq 0$, then 
		\[
			(d_u(n), -N)_p = 
			\begin{cases} 
				1& \textup{if }p\neq \ell,\\
				-1& p = \ell,
			\end{cases}
		\]
		for some $\delta,n$ as above.  
		
		We first prove that if $n$ satisfies the above assumptions, then $d_u(n)$ is negative.  Since $K$ is a totally imaginary extension of $F$, the relative discriminant of $\eta$ is negative under both real embeddings of $F\hookrightarrow \R$.  Using the definition of $\alpha_i, \beta_i$ and $c_K$, one can check that
		\[
			\Disc_{K/F}(\eta) = c_K + \alpha_1^2\frac{D}{2} + \sqrt{D}
			\left(\alpha_0\alpha_1 + \alpha_1^2\frac{D}2 - 2\beta_1\right).
		\]
		Recall that $\Norm_{F/\Q}(\Disc_{K/F}(\eta)) = \Dtilde$, thus $c_K + \alpha_1^2\frac{D}{2} < -\sqrt\Dtilde$.  Now consider
		\[
			d_u(n) = (\alpha_1\delta)^2 + \frac{2\delta(n + c_K\delta)}{D}
			= \frac{2\delta^2}{D}\left(\alpha_1^2 \frac{D}2 + \frac{n}\delta + c_K\right).
		\]
		Since $\delta^2\Dtilde - n^2 >0$, $\frac{n}{\delta}$ is bounded above by $\sqrt{\Dtilde}$.  Thus $d_u(n) < \frac{2\delta^2}{D}\left(\alpha_1^2 \frac{D}2 + \sqrt\Dtilde + c_K\right)$.  We have already shown that $\alpha_1^2 \frac{D}2 + \sqrt\Dtilde + c_K<0$ and $2\delta^2/D$ is clearly positive, so $d_u(n)$ is strictly negative. 
		
		Since $N$ is assumed to be positive, $(d_u(n), -N)_{\infty} = -1$, and so, by the product formula, there exists some prime $p$ such that $(d_u(n), -N)_p = -1$.  If $p \neq \ell$, then by Theorem~\ref{thm:LVSingularModuli}, $\mathscr{J}(d_uf_u^{-2}, d_x, t) = 0$ for all $f_u \in \Z$.  Another application of Theorem~\ref{thm:MainResult} shows that this implies that $(\CM(K), \textup{G}_1)_{\ell} = 0$.
	\end{proof}

	\subsection{An application: Denominators of Igusa class polynomials}
	\label{sec:Igusa}

		One of the important applications of the results in this paper is the computation of Igusa class polynomials.  Igusa invariants and Igusa class polynomials are the genus $2$ analogues of the $j$-invariant and the Hilbert class polynomial in genus $1$.  More precisely, Igusa invariants $i_1, i_2, i_3$ generate the function field of the coarse moduli space of smooth genus $2$ curves, and the Igusa class polynomials $H_{j, K},$ for $ j = 1,2,3$, are polynomials whose roots are Igusa invariants of genus $2$ curves $C/\C$ with an embedding $\iota\colon\OO_K\hookrightarrow \End(\Jac(C))$.  If a genus $2$ curve $C$ has CM by $K$, then $C$ is defined over $\Qbar$ and all of the Galois conjugates of $C$ also have CM by $K$.  Thus, $H_{j,K}\in\Q[z]$ for all $j$.
	
	However, in contrast to the genus $1$ case, the coefficients of $H_{j,K}$ are \emph{not} integral.  Therefore, in order to recover the coefficients from a complex or $p$-adic approximation, one needs more information on the denominators.  The denominators of the coefficients of $H_{j, K}$ divide a (known) multiple of the arithmetic intersection number  $\CM(K).\textup{G}_1$ (using multiplicative notation)~\cite{GL, GL10, Yang1}.  For a precise statement of this divisibility, see~\cite[\S9]{Yang1}.

Since Theorems~\ref{thm:main} and~\ref{thm:UpperBound} give a multiple of and, in many cases, an exact formula for $(\CM(K).\textup{G}_1)_{\ell}$, we obtain a formula for a multiple of the denominators of $H_{j,K}$ for \emph{all} primitive quartic CM fields.  Corollary~\ref{cor:Divisibility} also gives a restrictive characterization and bound on the primes that can appear in the denominators. 

	\subsection{Relationship to the Bruinier-Yang conjecture}\label{sec:WIN2}

		{As we mentioned in the introduction, earlier work of Yang proves an explicit formula for ${(\CM(K).G_1)_{\ell}}$, which was originally conjectured jointly with Brunier, under strong assumptions on the ramification in $K/\Q$.  The Brunier-Yang formula also sums over integers $\delta$ and $n$, satisfying similar conditions to those in Theorem~\ref{thm:main}.  However, the summand (for a fixed $\delta$ and $n$) is simpler than what appears in Theorem~\ref{thm:main}; in the Brunier-Yang formula, the summand only consists of a product of a valuation term and the number of ideals in a real quadratic field of a fixed norm.

		Under additional assumptions, the formula in Theorem~\ref{thm:main} simplifies to a formula which is strikingly similar to the Brunier-Yang formula. }
	
		
		\begin{thm}\label{thm:SpecialCase}
			Assume $(\dagger)$, that $\ell\nmid\delta$ for any positive integer such that $D - 4\delta$ is a square, and that $d_u(n)$ is a fundamental discriminant for any $n\in\Z$ such that $N = \frac{\delta^2\Dtilde - n^2}{4D}\in\ell\Z_{>0}$ and $2D|(n + c_K\delta)$.  Then
			\[
				\frac{(\CM(K).G_1)_{\ell}}{\log \ell} = 
					\sum_{\substack{\delta \in \Z_{>0}\\ 
						D - 4\delta=\square}}C_{\delta}
						\sum_{\substack{n\in\Z\\
						\frac{\delta^2\Dtilde - n^2}{4D} \in\ell\Z_{>0}\\
							2D\mid(n + c_K\delta) }}
						\mu_{\ell}(n)
						\tilde\rho_{d_u(n)}(N)\tilde\frakA_{d_u(n)}(N\ell^{-1}),
			\]
			where $C_{\delta} = \frac12$ if $D = 4\delta$ and $C_{\delta} = 1$ otherwise,
			\[
			\mu_{\ell}(n) = 
			\begin{cases}
				 v_{\ell}(N) & 
				\textup{if }\ell|\textup{gcd}(d_x(n),d_u(n)), \\
				\frac{v_{\ell}(N) + 1}2 & \textup{otherwise},
			\end{cases}
			\quad
			\tilde\rho_{d}(M) = 
			\begin{cases}
				0 & \textup{if }\left(d, -M \right)_p = -1 \\
				&\textup{for some }p|d, p\ne \ell\\
				2^{\#\{p|\textup{gcd}(d, M): p\ne \ell \}} & \textup{otherwise,}
			\end{cases}
		\]
		and $\tilde\frakA_d(M) := \#\{\frakb\subseteq \Z[\frac{d + \sqrt{d}}{2}] : \Norm(\frakb) = M\}$. 
		\end{thm}
		
		\bvchange{This similarity raises the question of whether there is a direct proof (i.e., without passing through the arithmetic intersection numbers ${(\CM(K).G_1)_{\ell}}$ relating the formula in Theorem~\ref{thm:SpecialCase} to the Brunier-Yang formula.  In recent work, the present authors and Anderson, Balakrishnan, and Park~\cite{WIN2} have shown that this is indeed the case, and give a direct proof that the formulas agree.
		} 


\section{Proof of Theorem~\ref{thm:main}}\label{sec:ProofOfMainTheorem}
	
	Since $K$ does not contain an imaginary quadratic field, $\CM(K)$ and $\textup{G}_1$ intersect properly~\cite[\S 3]{Yang2} and so
	\begin{equation}\label{eq:intersection}
		\frac{(\CM(K).G_1)_{\ell}}{\log \ell} = \sum_{P \in (\CM(K)\cap G_1)(\Fbar_{\ell})} \frac{1}{\#\Aut(P)}\cdot \textup{length } \widetilde{\OO}_{G_1\cap\CM(K), P}
	\end{equation}
where $\widetilde{\OO}_{G_1\cap\CM(K), P}$ is the local ring of  $G_1\cap\CM(K)$ at $P$.
		
	The cycle $G_1$ parametrizes products of elliptic curves with the product polarization; the Rosati involution induced by this polarization is given by
	\[
		g = \begin{pmatrix} g_{1,1} & g_{1,2}\\g_{2,1} & g_{2,2}\end{pmatrix}\in\End(E_1\times E_2) \mapsto
		g^{\vee} = \begin{pmatrix} g_{1,1}^{\vee} & g_{2,1}^{\vee}\\g_{1,2}^{\vee} & g_{2,2}^{\vee} \end{pmatrix},
	\]
	where $g_{i,j}\in\Hom(E_j,E_i)$ and $g_{i,j}^{\vee}$ denotes the dual isogeny of $g_{i,j}$~\cite[Section 3]{GL}.  Given this definition, one can see that a pair of elliptic curves $(E_1, E_2)$, together with an embedding $\iota\colon \OO_K \into \End(E_1\times E_2)$ that satisfies $\iota(\overline{\alpha}) = \iota(\alpha)^{\vee}$, determines a point $P\in(\CM(K)\cap G_1)(\Fbar_{\ell})$.  Conversely, a point $P\in (\CM(K)\cap G_1)(\Fbar_{\ell})$ determines an \emph{isomorphism class} $[(E_1,E_2,\iota)]$; we say two tuples
	$(E_1, E_2, \iota\colon \OO_K \hookrightarrow\End(E_1\times E_2))$ and
	$(E'_1, E'_2, \iota'\colon \OO_K \hookrightarrow\End(E'_1\times E'_2))$ are isomorphic if there exists an isomorphism $\psi\colon E_1 \times E_2 \stackrel{\sim}{\to} E'_1 \times E'_2$ such that
	\[
		\psi\circ\iota(\alpha) = \iota'(\alpha)\circ\psi\;\;
		\forall\alpha \in \OO_K,
		\;\textup{ and }\;
		\psi\circ g^{\vee}\circ\psi^{-1} =
		\left(\psi\circ g\circ\psi^{-1}\right)^{\vee} \;\;
		\forall g\in\End(E_1\times E_2).
	\]
	When $E_i = E'_i$, then the tuples are isomorphic if and only if there exists a $\psi \in \Aut(E_1 \times E_2)$ such that $\psi\circ\iota(\alpha) = \iota'(\alpha)\circ\psi$ for all $\alpha \in \OO_K$ and $\psi\psi^{\vee} = 1$.

	Given two elliptic curves $E_1, E_2$, let $\WW[[t_1, t_2]]$ be the deformation space of $E_1, E_2$, and let $\EE_1, \EE_2$ be the universal curves over this space.  We let $I_{E_1, E_2, \iota} \subset \WW[[t_1, t_2]]$ denote the minimal ideal such that there exists an $\widetilde{\iota}\colon \OO_K \into \End_{\WW[[t_1, t_2]]/I_{E_1, E_2, \iota}}(\EE_1, \EE_2)$ that agrees with $\iota$ after reducing modulo the maximal ideal of $\WW[[t_1, t_2]]$.  Then we have
	\[
		\textup{length }\widetilde{\OO}_{G_1\cap\CM(K), P} =
		 \textup{length } \WW[[t_1, t_2]]/I_{E_1,E_2,\iota},
	\]
	for any point $P\leftrightarrow(E_1, E_2, \iota)\in\left(G_1\cap\CM(K)\right)(\Fbar_{\ell})$.  Thus,~\eqref{eq:intersection} can be rewritten as
	\begin{equation}\label{eqn:intersection-emb}
		\frac{(\CM(K).G_1)_{\ell}}{\log \ell} = \sum_{(E_1,E_2,\iota)/\sim}
		\frac{1}{\#\Aut(E_1,E_2, \iota)} \cdot \textup{length }
		\frac{\WW[[t_1, t_2]]}{I_{E_1,E_2,\iota}},
	\end{equation}
	where $\Aut(E_1, E_2, \iota) := \left\{\sigma\in \Aut(E_1 \times E_2) : \sigma \iota(\alpha)\sigma^{\vee} = \iota(\alpha) \;\forall \alpha\in\OO_K\textup{ and } \sigma\sigma^{\vee} = 1\right\}$.  The condition  $\sigma\sigma^{\vee} = 1$ ensures that $\sigma$ preserves the product polarization.

	Since $\OO_K = \OO_F[\eta]$, giving an embedding $\iota\colon\OO_K\into\End(E_1\times E_2)$ is equivalent to specifying the image of $\omega = \frac12(D + \sqrt{D})$ and $\eta$, i.e., specifying two elements $\Lambda_1, \Lambda_2$ in $\End(E_1 \times E_2)$ such that
	\[
		\Lambda_1\Lambda_2  = \Lambda_2\Lambda_1,\quad
		\Lambda_2 + \Lambda_2^{\vee}  = \alpha_0 + \alpha_1\Lambda_1,\quad
		\Lambda_2\Lambda_2^{\vee}  =
			\beta_0 + \beta_1\Lambda_1, \textup{ and }\quad
		\Lambda_1^2 - D\Lambda_1 + \frac{D^2 - D}4  = 0.
	\]
The equivalence is obtained by letting $\Lambda_1 = \iota\left(\frac{D + \sqrt{D}}{2}\right), \Lambda_2 = \iota(\eta)$.  
This equivalence is a more precise reformulation of the Embedding Problem than the version used in~\cite[p. 463]{GL}, where the elements from $\OO_K$ being embedded were of a simpler form and were not necessarily generators of $\OO_K$.
By representing elements in $\End(E_1\times E_2)$ as $2\times2$ matrices $(g_{i,j})$ where $g_{i,j}\in\End(E_j, E_i)$ and expanding the above relations, we see that
	\[
		\Lambda_1 = \begin{pmatrix} a& b \\ b^{\vee} & D-a\end{pmatrix},
		\quad
		\Lambda_2 = \begin{pmatrix} x & y \\\alpha_1b^{\vee} - y^{\vee} & z
	\end{pmatrix},
	\]
	where $a$ is an integer and $x,b,y,z$ satisfy
	\begin{equation}\label{eq:xybz}
		\begin{array}{rl}
			\delta := \Norm(b) &= \frac{D - (D - 2a)^2}{4},\\
			\Tr(yb^{\vee}) = \Tr(y^{\vee}b) & =  \Norm(b)\alpha_1,\\
			\Norm(z) + \Norm(y) & =	\beta_0 + (D - a)\beta_1,\\
			\Norm(x) + \Norm(y) & = \beta_0 + a \beta_1,							\end{array}
		\quad
		\begin{array}{rl}
			\Tr(x) &=  \alpha_0 + a\alpha_1,\\
			\Tr(z) & = \alpha_0 + (D - a)\alpha_1,\\
			\beta_1b & = \alpha_1xb - xy + yz^{\vee},\\
			bz & = xb + (D - 2a)y.
		\end{array}
	\end{equation}
	After possibly conjugating $\Lambda_1, \Lambda_2$ by $\begin{pmatrix} 0& 1\\1&0\end{pmatrix}$ and interchanging $E_1, E_2$, we may assume that $2a \leq D$.  Then $a$ is uniquely determined by $\delta$.  Thus for a fixed $\delta$, the embedding $\iota$ is determined by a tuple $(x,y,b,z)$ satisfying the above relations.  Define $I := I_{x,y,b,z}\subseteq \WW[[t_1, t_2]]$ to be the minimal ideal such that there exists
	\[
		\widetilde{x}\in\End_{\WW[[t_1, t_2]]/I}(\EE_1), \quad
		\widetilde{y}, \widetilde{b}
			\in\Hom_{\WW[[t_1, t_2]]/I}(\EE_2, \EE_1), \quad\textup{ and }
		\widetilde{z}\in\End_{\WW[[t_1, t_2]]/I}(\EE_2)
	\]
	that reduce to $x,y,b,$ and $z$, respectively, modulo the maximal ideal of $\WW[[t_1, t_2]]$.  Then it is clear from the definition of $(x,y,b,z)$ that
	\[
		\textup{length }\frac{\WW[[t_1,t_2]]}{I_{E_1,E_2,\iota}} = 
		\textup{length }\frac{\WW[[t_1,t_2]]}{I_{x,y,b,z}}.
	\]

	Motivated by the definition of isomorphisms of triples $(E_1, E_2, \iota)$ that was given above, we say that two such tuples $(x,y,b,z), (x',y',b',z')$ are isomorphic if
	\[
		x\phi_1 = \phi_1 x',\; b\phi_2 = \phi_1b',\; y\phi_2 = \phi_1y',\;
		  z\phi_2 = \phi_2 z', \quad\textup{for some }\phi_i\in\Aut(E_i).
	\]
	In particular,
	\begin{align*}
		\Aut(x,y,b,z) & := \left\{
		\phi_i\in \Aut(E_i)
		: x\phi_1 = \phi_1 x,\; b\phi_2 = \phi_1b,\; y\phi_2 = \phi_1y,\;
		  z\phi_2 = \phi_2 z \right\}.
	\end{align*}
	If $4\delta \neq D$, then $(x,y,b,z)$ is isomorphic to $(x',y',b',z')$ if and only if the corresponding embeddings are isomorphic. Thus, $\#\Aut(x,y,b,z) = \#\Aut(E_1,E_2,\iota)$.
		
	If $4\delta = D$, then this no longer holds.  If $E_1 \neq E_2$, then $\#\Aut(x,y,b,z) = \#\Aut(E_1,E_2,\iota)$ for all $\iota$ and corresponding $x,y,b,z$; however, $(x,y,b,z)$ and $(z,y^{\vee},b^{\vee},x)$ correspond to the same embedding, although \bvchange{as tuples they are not isomorphic.}  If $E_1 = E_2$, then for each tuple $(x,y,b,z)$ we have two possibilities.  Either there exists an $(x',y',b',z')$ that is \emph{not} isomorphic to $(x,y,b,z)$ but corresponds to an isomorphic embedding, or there are twice as many automorphisms of $(E_1,E_2,\iota)$ as there are of $(x,y,b,z)$, where $\iota$ is the corresponding embedding.  In all cases, we see that for a fixed $\delta$
	\[
		\sum_{E_1,E_2,\iota}\frac{1}{\#\Aut(E_1,E_2,\iota)}\cdot
		\textup{length }\frac{\WW[[t_1,t_2]]}{I_{E_1,E_2,\iota}} = 
		C_{\delta}\sum_{\substack{E_1,E_2\\x,y,b,z}}
		\frac{1}{\#\Aut(x,y,b,z)} \cdot \textup{length }
		\frac{\WW[[t_1, t_2]]}{I_{x,y,b,z}},
	\]
	where $C_{\delta} = \frac12$ if $4\delta=D$ and $1$ otherwise.

	Fix $\delta, E_1, E_2$, and assume that there exists a tuple $(x,y,b,z)$ as above.  Then, there exists $x, u := yb^{\vee} \in \End(E_1)$ satisfying
	\begin{equation}\label{eq:xu}
		\begin{array}{rlcrl}
			\Tr(u) & =  \alpha_1\delta,&\quad&\Tr(x) & = \alpha_0 + a\alpha_1,\\
			\left(D - 2a\right)\Norm\left(u\right)
				+ \delta\Tr\left(xu^{\vee}\right) & =  \beta_1\delta^2,&&
		\delta\Norm(x) + \Norm(u) & =
			\delta\left(\beta_0 + a\beta_1\right),
		\end{array}
	\end{equation}
	where $a \in \Z$ is such that $a\leq D/2$ and $(D - 2a)^2 = D - 4\delta.$  This is easy to check using the relations~\eqref{eq:xybz} on $(x,y,b,z)$.  Let $I_{x,u}\subseteq\WW[[t_1]]$ be the minimal ideal such that there exists
	\[
		\tilde{x}, \tilde{u}\in\End_{\WW[[t_1]]/I_{x,u}}(\mathbb{E}_1)
	\]
	that reduce to $x,u$ respectively modulo the maximal ideal of $\WW[[t_1]]$.
	
	The remainder of the proof breaks into four steps.
	\begin{enumerate}
		\item[(\S\ref{subsec:xu})] Compute $\sum_{(E,x,u)}\textup{length }\frac{\WW[[t_1]]}{I_{x,u}}$, where the sum ranges over isomorphism classes of $(E,x,u)$ satisfying~\eqref{eq:xu},
		\item[(\S\ref{subsec:xybz})] For a fixed $(E, x, u)$ determine the number of isomorphism classes of $(E', y, b, z)$ such that $u = yb^{\vee}$ and $(x,y,b,z)$, satisfy~\eqref{eq:xybz} ,
		\item[(\S\ref{subsec:auto})] Calculate $\#\Aut(x,y,b,z)$ .
		\item[(\S\ref{subsec:mult})] Determine how the length of $\frac{\WW[[t_1, t_2]]}{I_{x,y,b,z}}$ relates to the length of $\frac{\WW[[t_1]]}{I_{x,u}}$. 
	\end{enumerate}
	As it is not necessarily obvious how the arguments in \S\S\ref{subsec:xu}--\ref{subsec:mult} come together, we summarize the argument in \S\ref{subsec:summary}.
		
	\subsection{Calculating the number of $(E, x, u)$}\label{subsec:xu}
		In this section we will compute
		\[
			\sum_{\substack{(E,x,u) \textup{ satisfying}\\ \eqref{eq:xu}}} \textup{length }\frac{\WW[[t_1]]}{I_{x,u}},
		\]
		where the sum ranges over one representative from each isomorphism class; we say that $(E,x,u)$ is isomorphic to $(E', x', u')$ if there exists an isomorphism $\psi\colon E \to E'$ such that $\psi\circ x = x'\circ \psi$ and $\psi\circ u = u'\circ \psi$.  
		
		First we show that the elements $(E,x,u)$ are naturally partitioned by an integer $n$ and that $E$ is always supersingular.
		\begin{prop}\label{prop:ell-divisibility}
			Let $E$ be an elliptic curve over $\Fbar_{\ell}$ and assume that there exists endomorphisms $x$ and $u$ of $E$ that satisfy~\eqref{eq:xu}.  Then $E$ is supersingular and there exists an integer $n$ such that
			\begin{equation}\label{eq:conditions_on_n}
				\frac{\delta^2\Dtilde - n^2}{4D}\in \ell\Z_{>0},
				\textup{ and } n + c_{K}\delta \equiv 0 \pmod{2D},
			\end{equation}
			where $c_K = \alpha_0^2 + \alpha_0\alpha_1D + \alpha_1^2\frac{D^2 - D}{4} - 4\beta_0 - 2\beta_1D$ \bvchange{as in~\eqref{eq:cK}.}
		\end{prop}
		\begin{proof}
			Let $\Rtilde:= \Z\oplus\Z x\oplus\Z u\oplus\Z xu^{\vee}$ denote the sub-order of $\End(E)$ generated by $x$ and $u$ and for any element $v\in\End(E)$, write $\Disc(v) := \Tr(v)^2 - 4\deg(v)$ for the discriminant of the element.  A straightforward calculation shows that the discriminant of $\Rtilde$ is $\left(\frac{\Disc(x)\Disc(u) - (\Tr(x)\Tr(u) - 2\Tr(xu^{\vee}))^2}{4}\right)^2$ and that
			\[
				\Disc(x)\Disc(u) - (\Tr(x)\Tr(u) - 2\Tr(xu^{\vee}))^2 =
				-\Disc(2xu^{\vee} - \Tr(u)x + \Tr(x)u).
			\]
			Since the discriminant of any endomorphism of $E$ is non-positive, we conclude that
			\[
				\frac{\Disc(x)\Disc(u) - (\Tr(x)\Tr(u) - 2\Tr(xu^{\vee}))^2}{4}
			\]
			is a non-negative integer.  Now let $n := \frac{-2D\Norm(u)}{\delta}-\delta c_K$.  An easy, although tedious, computation shows that
			\begin{equation}\label{eq:Nreln}
				\frac{\delta^2\Dtilde - n^2}{4D} =
				\frac{\Disc(x)\Disc(u) - (\Tr(x)\Tr(u) - 2\Tr(xu^{\vee}))^2}{4}.
			\end{equation}
			Since $K$ does not contain an imaginary quadratic field, $\Dtilde$ is not a square, and so this quantity must be strictly positive.   This implies that $\Rtilde$ is rank $4$ and so we conclude that $E$ is supersingular and $\Rtilde$ is a suborder in $\BBl$, the quaternion algebra ramified only at $\ell$ and infinity.  Since $\ell$ divides the discriminant of any order in $\BBl$, we have $\delta^2\Dtilde - n^2 \in 4D\ell\Z_{>0}$.  This completes the proof of the first assertion.  The second assertion follows since
			\[
				\frac{n + c_{K}\delta}{2D} =
				\frac{-\Norm(u)}{\delta} = \Norm(x) - \beta_0-a\beta_1\in \Z.
			\]
		\end{proof}
		\begin{remark}
			In~\cite[p.465]{GL}, Goren and the first author proved that $E$ must be supersingular if $K$ does not contain an imaginary quadratic field.  Proposition~\ref{prop:ell-divisibility} gives another proof of this result.
		\end{remark}
		Proposition~\ref{prop:ell-divisibility} shows that the tuples $(E, x, u)$ satisfying~\eqref{eq:xu} can be partitioned by integers $n$ satisfying~\eqref{eq:conditions_on_n}.  By the proof of Proposition~\ref{prop:ell-divisibility}, fixing such an $n$ implies that $\Norm(u) = n_u(n), \Norm(x)= n_x(n),$ and $ \Tr(xu^{\vee}) = t_{xu^{\vee}}(n)$ where
		\[
			n_u(n)  = \frac{-\delta(n + c_K\delta)}{2D},\;\;
			n_x(n)  = \beta_0 + a\beta_1 - \frac{n_u(n)}{\delta},\;\;\&\;\;
			t_{xu^{\vee}}(n) = \beta_1\delta - (D- 2a) \frac{n_u(n)}{\delta},
		\]
		as in~\eqref{eq:normu}--\eqref{eq:trxu}.
		
		The trace of $x$ and $u$ are already determined by $\delta$, so we  define $d_u(n) := (\alpha_1\delta)^2 - 4n_u(n)$ and $d_x(n) := (\alpha_0 + a\alpha_1)^2 - 4n_x(n)$.  For the rest of the section, we assume that $n$ is a fixed integer satisfying~\eqref{eq:conditions_on_n}. We define
		\[
			\calE = \calE(n) := \left\{
			\begin{array}{rl}
				[(E,x,u)]:& \Tr(x) = \alpha_0 + a\alpha_1, \Tr(u) = \alpha_1\delta,\\
				&\Norm(u) = n_u(n), \Norm(x) = n_x(n), \Tr(xu^{\vee}) = t_{xu^{\vee}}(n)
			\end{array}\right\},
		\]
		where $[(E,x,u)]$ denotes the isomorphism class of $(E,x,u)$.
		We claim that the length of $\WW[[t]]/I_{x,u}$ is constant for all $(E,x,u)\in\calE$.
		\begin{theorem}\label{thm:multiplicity}
			Let $(E,x,u) \in \calE$.  Then 
			\[
				\textup{length }\WW[[t]]/I_{x,u} = 
					\begin{cases}
						v_{\ell}(\frac{\delta^2\Dtilde - n^2}{4D}) & 
						\textup{if }\ell|\textup{gcd}(d_u(n), d_x(n)),\\ 
						\frac12\left(v_{\ell}(\frac{\delta^2\Dtilde-n^2}{4D}) 
							+ 1\right) & \textup{otherwise}.
					\end{cases}
			\]
		\end{theorem}
		\begin{proof}
			First we show that $\calE\neq\varnothing$ only if at least one of $d_u(n), d_x(n)$ is the discriminant of a quadratic imaginary order that is maximal at $\ell$.
			\begin{lemma}\label{lem:RelativelyPrimeIndices}
				Let $E$ be a supersingular elliptic curve over $\Fbar_{\ell}$ and let $x, u \in \End(E)$ be endomorphisms satisfying~\eqref{eq:xu}.  Then the indices 
				\[
				[\Q(x)\cap\End(E):\Z[x]]\textup{ and }[\Q(u)\cap\End(E):\Z[u]]
				\]
				are relatively prime.  In particular, at least one of $\Z[x]$, $\Z[u]$ is a quadratic imaginary order maximal at $\ell$.
			\end{lemma}
			\begin{proof}
				Define $w := x + (D - 2a)\frac{u}{\delta}.$ The conditions~\eqref{eq:xu} on $x,u$ imply that
				\[
					\begin{pmatrix} 1 & 0 \\ 0 & 1 \end{pmatrix}, \quad
					\begin{pmatrix} a & \delta \\ 1 & D - a\end{pmatrix}, \quad
					\begin{pmatrix} x & u \\ u/\delta & w \end{pmatrix}, 
						\quad\textup{ and }
					\begin{pmatrix} ax + u & \delta x + (D- a)u\\
						x + (D - a)u/\delta  & (D - a)w + u
					\end{pmatrix}	
				\]
				generate a rank $4$ $\Z$-submodule $\widetilde{S} \subseteq \Mat_2(\BBl)$ that is isomorphic to $\OO_K$ (the isomorphism sends the above matrices to $1, \frac12(D + \sqrt{D}), \eta,$ and $ \frac12(D+ \sqrt{D})\eta$, respectively).  Let $p$ be a prime and let $S$ be any order in $\Mat_2(\BBl\otimes_{\Q}\Q_{p})$ that contains $\widetilde{S}$.  Since $\OO_K$ is the unique maximal order of $K$, an integral combination of the matrices above can only be in $p S$ if every coefficient is divisible by $p$.  We will show that if $p$ divides both $[\Q(x)\cap\End(E):\Z[x]]$ and $ [\Q(u)\cap\End(E):\Z[u]]$, then some $p$-primitive integral combination of the above matrices is in $p \Mat_2(\End(E))$, thus arriving at a contradiction.

				If $p$ divides $[\Q(x)\cap\End(E):\Z[x]]$ and $ [\Q(u)\cap\End(E):\Z[u]]$, then
				\[
					\frac{2p x - p\Tr(x) + \Disc(x)}{2p}, \quad
					\frac{2p u - p\Tr(u) + \Disc(u)}{2p},
				\]
				are both in $p\End(E)$.  Consider the $p$-primitive combination
				\[
					\frac{\Disc(u) - p\Tr(u)}{2p} + \left[a - D + 
					\begin{pmatrix} a & \delta\\ 1 & D - a\end{pmatrix}\right]\cdot
					\left[\frac{\Disc(x) - p\Tr(x)}{2p} + 
					\begin{pmatrix} x & u \\ u/\delta & w \end{pmatrix}\right].
				\]
				After expanding and rearranging terms, we can express this $p$-primitive combination as
				\[
					\frac{2p u - p\Tr(u) + \Disc(u)}{2p} + 
					\frac{2p x - p\Tr(x) + \Disc(x)}{2p}
					\begin{pmatrix}2a - D & \delta\\1 & 0\end{pmatrix},
				\]
				which is clearly in $p\Mat_2(\End(E)\otimes\Z_p).$  This completes the proof of the first statement.  By~\cite[Chap. II, Lemma 1.5]{Vigneras} $\End(E)\otimes\Z_{\ell}$ consists of all integral elements in $\End(E)\otimes\Q_{\ell}$ so both $\Q(u)\cap\End(E)$ and $\Q(x)\cap\End(E)$ are orders that are maximal at $\ell$.  Since at most one of $[\Q(x)\cap\End(E):\Z[x]]$ and $ [\Q(u)\cap\End(E):\Z[u]]$ are divisible by $\ell$, at least one of $\Z[x]$ and $\Z[u]$ is maximal at $\ell$, as desired.
			\end{proof}

		Now we return to the proof of Theorem~\ref{thm:multiplicity}.  Let $d_1\in\{d_x, d_u\}$ be such that $d_1$ is the discriminant of a quadratic imaginary order that is maximal at $\ell$ and such that $d_1$ has minimal $\ell$-valuation; this is possible by the preceding lemma.  Let $\omega_1\in\{\frac12(d_u - t_u) + u, \frac12(d_x - t_x) + x\}$ be such that $\omega_1$ has discriminant $d_1$.  We define $d_2,$ and $ \omega_2$ to be such that
		\[
			\{d_1, d_2\} = \{d_u, d_x\}, \quad\textup{and }
			\{2\omega_1,2\omega_2\} = 
			\{d_u - t_u + 2u, d_x - t_x + 2x\}.
		\]
		From these definitions, it is clear that $I_{x,u} = I_{\omega_1,\omega_2}$.
		
		Work of Gross~\cite{Gross-CanonicalLifts} shows that $\WW[[t]]/I_{\omega_1}$ is isomorphic to $\WW_{d_1}$, the ring of integers in $\Q_{\ell}(\sqrt{d_1})^{\textup{unr}}$.  An explicit description of $\End_{\WW_{d_1}/\mm^k}(\EE\bmod I_{\omega_1})$ (where $\mm$ is the unique maximal order of $\WW_{d_1}$) is given in~\cite[\S6]{LV-SingularModuli}, for all $k$.  Using this description and~\cite[Proof of Thm. 3.1]{LV-SingularModuli}, we see that $\omega_2\in\End_{\WW_{d_1}/\mm^k}(\EE\bmod I_{\omega_1})$ if and only if $\ell^r$ divides
		\begin{equation}\label{eq:ReducedDisc}
			\frac{d_1d_2 - (d_1d_2 - 2\Tr(\omega_1\omega_2^{\vee}))^2}{4} =
			\frac{d_xd_u - (t_xt_u - 2t_{xu^{\vee}}(n))^2}{4},
		\end{equation}
		where $r = k$ if $\ell| d_1$ and $r = 2k - 1$ otherwise.  By the proof of Proposition~\ref{prop:ell-divisibility}, the quantity in~\eqref{eq:ReducedDisc} is equal to $(\delta^2\Dtilde - n^2)/(4D).$  Since the length of $\WW[[t]]/I_{\omega_1,\omega_2}$ is equal to the maximum $k$ such that $\omega_2\in\End_{\WW_{d_1}/\mm^k}(\EE\bmod I_{\omega_1})$ this completes the proof.
	\end{proof}
	\begin{cor} The sum $\sum_{\substack{(E,x,u) \textup{ satisfying}\\ \eqref{eq:xu}}} \textup{length }\frac{\WW[[t_1]]}{I_{x,u}}$ equals
		\[
			\sum_{\substack{n\in\Z\\\delta^2\Dtilde - n^2\in4D\ell\Z_{>0}\\
			2D|(n + c_K\delta)}} \#\calE(n)\cdot 
			\begin{cases}
				v_{\ell}(\frac{\delta^2\Dtilde - n^2}{4D}) & 
				\textup{if }\ell|\textup{gcd}(d_u(n), d_x(n)),\\
				\frac12\left(v_{\ell}(\frac{\delta^2\Dtilde-n^2}{4D}) 
					+ 1\right) & \textup{otherwise}.
			\end{cases}
		\]
	\end{cor}
	
	The remainder of the section will be devoted to the proof of the following proposition.
	\begin{prop}
		Let $n\in\Z$ be such that $\delta^2\Dtilde - n^2 \in 4D\ell\Z_{>0}$ and $2D|(n + c_K\delta)$.  Then
		\[
			\#\calE(n) = \sum_{f_u\in\Z_{>0}} \mathscr{J}(d_u(n)f_u^{-2}, d_x(n), t(n, f_u)).
		\]
	\end{prop}
	\begin{proof}
		Recall that $\mathscr{J}(d_1,d_2,t)$ equals
		\[
			\sum_{E/\Fbar_{\ell}}\#\left\{
			\begin{array}{ll}
				i_j\colon\Z\left[\frac{d_j + \sqrt{d_j}}2\right]\hookrightarrow 
				\End(E) : &
				\Tr(i_1(d_1 + \sqrt{d_1})i_2(d_2 - \sqrt{d_2})) = 4t,\\
				& i_1(\Q(\sqrt{d_1}))\cap\End(E) = \Z\left[\frac{d_1 + \sqrt{d_1}}2\right]
			\end{array}
			\right\}/\End(E)^{\times}
		\]
		where the sum ranges over isomorphism classes of elliptic curves.  Let $(E, x, u) \in \calE(n)$ and set $f_u := [\Q(u)\cap\End(E) : \Z[u]]$.  We let $d_1 := d_u(n)f_u^{-2}$ and $d_2 := d_x(n)$.  Define two embeddings
		\begin{align*}
			i_1\colon\Z\left[\frac{d_1 + \sqrt{d_1}}{2}\right] \to\End(E), &
			\quad \frac{d_1 + \sqrt{d_1}}{2} \mapsto 
			\frac1{2f_u^2}\left(2f_uu - f_u\alpha_1\delta + d_u(n)\right),\\
			i_2\colon\Z\left[\frac{d_2 + \sqrt{d_2}}{2}\right] \to\End(E), &
			\quad \frac{d_2 + \sqrt{d_2}}{2} \mapsto 
			\frac12\left(2x - (\alpha_0 + a\alpha_1) + d_x(n)\right).
		\end{align*}
		From the definition of $f_u$ and $d_j$, one can easily check that these maps are well-defined and that $i_1(\Q(\sqrt{d_1}))\cap\End(E) = \Z\left[\frac{d_1 + \sqrt{d_1}}2\right]$.  One also has
		\begin{align*}
			\Tr(i_1(d_1 + \sqrt{d_1})i_2(d_2 - \sqrt{d_2})) &= 
			\frac1{f_u^2}\Tr(\left(2f_uu - f_u\alpha_1\delta + d_u(n)\right)\left(2x^{\vee} - (\alpha_0 + a\alpha_1) + d_x(n)\right))\\
			& = \frac{4}{f_u}t_{xu^{\vee}}(n) - \frac{2}{f_u}t_xt_u + \frac{2d_u(n)d_x(n)}{f_u^2} = 4t(n, f_u),
		\end{align*}
		as desired.  It is clear that if $(E, x, u)$ and $(E, x', u')$ are isomorphic, then the corresponding embeddings described above differ by conjugation by an element of $\End(E)^{\times}$.  This completes the proof.
	\end{proof}

	\subsection{Determining the pre-image of $(E, x,u)$}\label{subsec:xybz}
		In this section we prove the following theorem.

		\begin{thm} \label{thm:PreImage}
			Let $E$ be a supersingular elliptic curve and assume there exists $x,u\in\End(E)$ satisfying~\eqref{eq:xu}.  Let $f_u\in\Z_{> 0}$ be such that $\Q(u)\cap\End(E)$ is an order of discriminant $d := \frac{\Disc(u)}{f_u^2}$.  Then
			\[
				\#\left\{(E', y, b, z) : u = yb^{\vee}, (x,y,b,z)
			   \textup{ satisfy }~\eqref{eq:xybz}\right\}
				 = \prod_{p|\delta, p\neq \ell}
				\left(
				\sum_{\substack{j = 0\\j\equiv v_p(\delta)\bmod2}}^{v_p(\delta)}
				 \frakI^{(p)}_{j - r_p}(\Tr(w), \Norm(w))\right) ,
			\]
			where $w := x + (D - 2a)u/\delta$, $r_p := \max\left(v_p(\delta) - \min(v_p(f_u), v_p(\frac{\Disc(u) - \Tr(u)f_u}{2f_u})), 0\right)$ and
			\[
				\frakI_{C}^{(p)}(a_1, a_0) = 
				\begin{cases}
					\#\{\widetilde{t} \bmod p^{C} : 
						\widetilde{t}^2 - a_1\widetilde{t} + a_0
						\equiv 0 \pmod{p^{C}}\}
					& \textup{ if }  C\geq 0,\\
					0 & \textup{ if } C< 0.
				\end{cases}
			\]
		\end{thm}
		\begin{proof}

		Fix an $(E,x,u)$ satisfying~\eqref{eq:xu}.  Assume that there exists an elliptic curve $E'$, $b, y \in \Hom(E', E)$, and $z\in \End(E')$ such that $u = yb^{\vee}$, $bz = xb + (D-2a)y$.  Then there is a left integral ideal $I := \Hom(E', E)\circ b^{\vee}$ of $R := \End(E)$ which has the following properties:
		\begin{enumerate}
			\item $\Norm(I) = \delta$,
			\item $\delta, u \in I$, and
			\item $w := x + (D - 2a)\frac{u}{\delta}\in \RO(I) := \{A\in R\otimes\Q : IA\subseteq I\}$.
		\end{enumerate}
		In fact, we claim that this map is a bijection (when $(E,y,b,z)$ are considered up to equivalence), so
		\[
					\#\left\{[(E', y, b, z)] : u = yb^{\vee}, (x,y,b,z)
				   \textup{ satisfying }~\eqref{eq:xybz}\right\}
					 =  \#\left\{I \subseteq R :
					\textup{ satisfying } (1), (2), (3)\right\}.
		\]
		
		The proof of this claim relies on Deuring's correspondence between supersingular elliptic curves and ideal is $\BBl$; we describe this now.  Fix a supersingular elliptic curve $E/\Fbar_{\ell}$, and fix an isomorphism $\psi:\End(E) \stackrel{\sim}{\to} R \subseteq \BBl$, where $R$ is a maximal order.  Note that $\psi$ allows us to view elements of $\End(E)\otimes\Q$ as elements of $\BBl$.  Given an element $\phi \in \Hom(E, E')$, we obtain an embedding $\Hom(E', E) \to \End(E)$ by mapping $f \mapsto f\circ \phi$.  Thus we can view $\Hom(E', E)$ as a left ideal of $\End(E)$ or, by using the isomorphism $\psi$, as a left ideal $I$ of $R$.  In fact, Deuring 
showed that the map
	\[
		\left\{(E', \phi: E \to E')\right\}
		\to \left\{\textup{left ideals }I \textup{ of }R\right\}, 
		\quad(E', \phi)\mapsto \psi(\Hom(E', E)\phi)
	\]
	is surjective.   In addition, if $\psi(\Hom(E', E)\phi') = \psi(\Hom(E'', E)\phi')$, then $\phi'' = \varphi'\circ\phi'$, for some $\varphi'\in \textup{Isom}(E', E'')$.  For a more complete description of this correspondence see Deuring's original article~\cite{Deuring} or~\cite[\S\S3,4]{Waterhouse}.
		
	The morphism $\psi$ also allows us to view $\End(E')$ as a subring of $\BBl$; fix an isogeny $\phi\colon E\to E'$, and consider the map $\psi'\colon\End(E')\to\BBl$ that sends an endomorphism $f$ to $\frac{1}{\deg(\phi)}\psi(\phi^{\vee}\circ f\circ \phi)$.  Let $R' = \psi'(\End(E'))$.  It is clear that $R'$ is contained in the right order of the ideal $I = \psi(\Hom(E', E)\phi)$, and since $R'$ is a maximal order we must have equality.
	
	Now we return to the proof of the claim.  Let $I\subseteq R$ be an ideal satisfying conditions $(1)$, $(2)$, and $(3)$.  Then, by the discussion above, there exists an elliptic curve $E'$ and an isogeny $\phi\colon E \to E'$.  Let $b := \phi^{\vee}$.  Since $I$ has norm $\delta$ (\bvchange{by condition $(1)$}), the degree of $b$ is also $\delta.$  Since $u \in I$ (\bvchange{by condition $(2)$}), there exists a $y \in \Hom(E', E)$ such that $yb^{\vee} = u$; moreover, $y$ is unique.  Since $x + (D-2a)u/\delta\in\RO(I)$ (\bvchange{by condition $(3)$}), there exists a $z\in\End(E')$ such that $bzb^{\vee}/\delta = x + (D  - 2a) u/\delta$, or rather that $bz = xb + (D - 2a)y$; one can check that this relation uniquely determines $z$.  Thus, given an $I$ that satisfies conditions $(1)$, $(2)$, and $(3)$, we obtain $(E', y, b, z)$ such that $u = yb^{\vee}$ and $(x,y,b,z)$ satisfy~\eqref{eq:xybz}.
	
	Let $E_1', E_2'$ be elliptic curves and $\phi_i\colon E\to E'_i$ isogenies  such that $\Hom(E_i', E)\phi_i = I$.  Define $b_i := \phi_i^{\vee}$.  Since $\Hom(E_1', E)\phi_1 = \Hom(E_2', E)\phi_2$, there exists some $\phi_{1,2}\in\textup{Isom}(E_1',E_2')$ such that $b_1 = b_2\circ\phi_{1,2}$.  As described above, there exists $y_i\in\Hom(E_i', E)$ and $z_i\End(E_i')$ that are unique such that 
	\[
		u = y_ib_i^{\vee}, \quad b_iz_i = xb_i + (D - 2a)y_i.
	\]
	Since $\tilde y_1 := y_2\circ\phi_{1,2}$ and $\tilde z_1 := \phi_{1,2}^\vee z_2 \phi_{1,2}$ also satisfy these equations, we have $y_1= \tilde y_1$ and $z_1 = \tilde z_1$.  Thus $(x,y_1,b_1,z_1)$ is isomorphic to $(x, y_2, b_2,z_2)$.  This completes the proof of the claim.

	Now we have reduced the problem to a question about ideals in $\BBl$.

	\begin{theorem}\label{thm:CountingIdeals}
		Fix $R$ a maximal order in $\BBl$.  Assume that $x,u\in R$ and $\gamma,\delta\in \Z$ are such that
		\begin{equation}\label{eq:Assumptions1}
			\Tr(u), \quad \Norm(u), \quad \textup{and }\quad
			\Tr(xu^{\vee}) + \gamma\Norm(u)/\delta \quad 
			\textup{ are }0 \textup{ modulo }\delta.
		\end{equation}
		Define $w := x + \gamma u/\delta$, $c_p\in\Z$ to be such that $up^{-c_p}\in R_p\setminus pR_p$, and $r_p := \max(v_p(\delta) - c_p, 0)$.  Assume that for all $p|\delta, p\ne\ell$, either $c_p = 0$ or $\Q_p(p^{r_p}w)\cap(\End(E)\otimes\Z_p) = \Z_p[p^{r_p}w].$
		Then $\#\left\{I\subseteq R : \delta,u\in I, \Norm(I) = \delta, 
		\textup{ and } w := x + \gamma u/\delta \in \RO(I)\right\}$ equals
		\[
			\prod_{p|\delta, p\ne\ell}\left(
			\sum_{\substack{j = 0 \\j \equiv v_p(\delta)\pmod 2}}^{v_p(\delta)} 
			\frakI^{(p)}_{j - r_p}(\Tr(w), \Norm(w))\right),
		\]
		where 
		\[
			\frakI_{C}^{(p)}(a_1, a_0) = 
			\begin{cases}
				\#\{\widetilde{t} \bmod p^{C} : 
					\widetilde{t}^2 - a_1\widetilde{t} + a_0
					\equiv 0 \pmod{p^{C}}\}
				& \textup{ if }  C\geq 0,\\
				0 & \textup{ if } C<0.
			\end{cases}
		\]
	\end{theorem}
	Since the proof of this theorem is completely independent of the rest of the paper, we defer it until~\S\ref{sec:IdealsInM2Zp}.  If we show that $x,u,\delta, \gamma = D - 2a$ satisfy the assumptions of Theorem~\ref{thm:CountingIdeals}, and that $c_p = \min\left(v_p(f_u), v_p(\frac{\Disc(u) - f_u\Tr(u)}{2f_u})\right)$, then we can apply Theorem~\ref{thm:CountingIdeals} to complete the proof of Theorem~\ref{thm:PreImage}  
	
	It is clear from~\eqref{eq:xu} that the assumptions listed in~\eqref{eq:Assumptions1} are satisfied; we now prove the claim regarding $p^{r_p}w$.
	
	\begin{lemma}\label{lem:OptEmbedding}
		Let $p$ be a prime such that $p|\delta$ and $c_p \neq 0$.  Then 
		\[
			\Q_p(p^{r_p}w)\cap (\End(E)\otimes\Z_p) = \Z_p[p^{r_p}w].
		\]
	\end{lemma}
	\begin{proof}
	  From the definition of $c_p$, it is clear that $\tilde w := p^{r_p}w\in (\End(E)\otimes\Z_p)$.  If $\Disc(\tilde w)$ has trivial conductor, then the result is immediate.  Assume that $\Disc(\tilde w)$ has non-trivial conductor.  Then $\Q_p(\tilde w)\cap (\End(E)\otimes\Z_p) \neq \Z_p[\tilde w]$ if and only if $\frac{\tilde w}{p} + \frac{\Disc(\tilde w) - p\Tr(\tilde w)}{2p^2}\in\End(E)\otimes\Z_p$.
	
		First assume that $r_p >0$.  Since $w$ is integral $p^2|\Disc(\tilde w)$ and $p|\Tr(\tilde w)$.  Thus 
		\[
			\frac{\tilde w}{p} + \frac{\Disc(\tilde w) 
			- p\Tr(w)}{2p^2}\in\End(E)\otimes\Z_p
		\] 
		if and only if $\frac{\tilde w}{p}\in\End(E)\otimes\Z_p$, which in turn is equivalent to $\frac1p(D - 2a)\frac{u}{p^{c_p}}\frac{p^{v(\delta)}}{\delta}\in\End(E)\otimes\Z_p$.  By definition of $c_p$, this occurs if and only if $p|D - 2a$. 
		 \bvchange{If $p$ is odd, then since $D$ is the discriminant of a real quadratic field we have that $v_p(D)\leq 1$. Therefore }either $p\nmid D - 2a$ or $v_p(\delta) = 1$.  \bvchange{If $p = 2$, then a similar calculation gives the same conclusion.}
		 However, if $r_p >0 $ and $c_p \neq 0$, then $v_p(\delta) \geq 2$.  Therefore, for all primes $p$, $p\nmid D - 2a$ and hence $\Q_p(p^{r_p}w)\cap (\End(E)\otimes\Z_p) = \Z_p[p^{r_p}w]$.  
		
		Now assume that $r_p = 0$ (so $w = \tilde w$). This case will be similar to the proof of Lemma~\ref{lem:RelativelyPrimeIndices}.  Consider the element
		\[
			\left[-a + \begin{pmatrix}a & \delta\\1 & D-a\end{pmatrix}\right]
			\cdot
			\left[\frac{\Disc(w) - p\Tr(w)}{2p} + 
			\begin{pmatrix}x & u\\u/\delta & w\end{pmatrix}\right] = 
			\begin{pmatrix}u & \delta w' \\
				w' & 
				u + (D - 2a)w'\end{pmatrix}
		\]
		in $\Mat_2(\End(E))$, where $w' = w + \frac{\Disc(w) - p\Tr(w)}{2p}$.  If $\Q_p(\tilde w)\cap (\End(E)\otimes\Z_p) \neq \Z_p[\tilde w]$, then this element is in $p\Mat_2(\End(E))$.  However, 
		\[
			\begin{pmatrix} 1 & 0 \\ 0 & 1 \end{pmatrix}, \quad
			\begin{pmatrix} a & \delta \\ 1 & D - a\end{pmatrix}, \quad
			\begin{pmatrix} x & u \\ u/\delta & w \end{pmatrix}, 
				\quad\textup{ and }
			\begin{pmatrix} ax + u & \delta x + (D- a)u\\
				x + (D - a)u/\delta  & (D - a)w + u
			\end{pmatrix}	
		\]
		generate a rank $4$ algebra that is isomorphic to $\OO_K$ so a $p$-primitive integral combination of these elements can never be in $p\Mat_2(\End(E))$.  Thus $\Q_p(\tilde w)\cap (\End(E)\otimes\Z_p) = \Z_p[\tilde w]$.
	\end{proof}
	
	Now we turn to the computation of $c_p$.  By the definition of $f_u$,
	\[
		\Q_p(u)\cap (\End(E)\otimes\Z_p) = 
		\Z_p\left[\frac{u}{f_u} + \frac{\Disc(u) - f_u\Tr(u)}{2f_u^2}\right].
	\]
	Since
	\[
		\frac{u}{p^s} = \frac{f_u}{p^s}
		\left(\frac{u}{f_u} + \frac{\Disc(u) - f_u\Tr(u)}{2f_u^2}\right)
		- \frac{\Disc(u) - f_u\Tr(u)}{2f_up^s},
	\]
	it is clear that $\frac{u}{p^s}\in\End(E)\otimes\Z_p$ if and only if $s\leq v_p(f_u)$ and $s\leq v_p(\frac{\Disc(u) - f_u\Tr(u)}{2f_u})$.  Since $c_p$ is the maximal $s$ such that $u/p^s\in\End(E)\otimes\Z_p$, this completes the proof of Theorem~\ref{thm:PreImage}.
	\end{proof}

	\subsection{Computing $\#\Aut(x,y,b,z)$}\label{subsec:auto}

	\begin{lemma}
		Fix elliptic curves $E_1, E_2$ and assume there exist isogenies $x\in\End(E_1), z\in\End(E_2),$ and $y,b\in\Hom(E_2,E_1)$ satisfying~\eqref{eq:xybz}.  Then $\#\Aut(x,y,b,z) = 2$.
	\end{lemma}
	\begin{proof}
		Recall that
		\begin{align*}
			\Aut(x,y,b,z) & := \left\{
			\phi_i\in \Aut(E_i)
			: x\phi_1 = \phi_1 x, b\phi_2 = \phi_1b, y\phi_2 = \phi_1y,
			  z\phi_2 = \phi_2 z \right\}.\\
		\Aut(x,u) & := \left\{
		\phi\in \Aut(E)
		: x\phi = \phi x, u\phi = \phi u \right\}.
		\end{align*}
		
		It is clear that there is a homomorphism $\Aut(x,y,b,z) \to \Aut(x, yb^{\vee}), \quad (\phi_1,\phi_2)\mapsto \phi_1$.  Similarly we obtain a homomorphism 
		\[
			\Aut(x,y,b,z)\to \Aut(z, b^{\vee}y):=
			\left\{\phi\in\Aut(E) : \phi z=z\phi, \phi b^{\vee}y = b^{\vee}y\phi\right\},
		\] 
		that sends $(\phi_1, \phi_2)\mapsto \phi_2.$  Therefore, we have an embedding 
		\[
			\Aut(x,y,b,z) \hookrightarrow 
			\Aut(x, u := yb^{\vee})\times\Aut(z, u^* := b^{\vee}y).
		\]

		The proof of Proposition~\ref{prop:ell-divisibility} shows that $x, u$ generate a sub-order of $\End(E_1)$ of finite index and that $\End(E_1)$ is rank $4$.  The same argument can be applied to $z, u^* = b^{\vee}y \in \End(E_2)$ to show that these elements generate a sub-order of $\End(E_2)$ of finite index and that $\End(E_2)$ is rank $4$.
		Thus, $\Aut(x,u)\subseteq Z(\End(E_1))^{\times}$ and $\Aut(z, u^*)\subseteq Z(\End(E_2)^{\times})$ where $Z(A)$ denotes the center of $A$.  Since the center of $\End(E_i)$ is just $\Z$, we see that $\Aut(x,u) = \Aut(z, u^*) = \{\pm1\}$.  Using the embedding above, it is easy to check that $\Aut(x,y,b,z) = \left\{\pm(1,1)\right\}$.
	\end{proof}
	\subsection{Relating multiplicities}\label{subsec:mult}

		Fix elliptic curves $E_1, E_2,$ and isogenies $x \in\End(E), y,b\in\Hom(E_2, E_1),$ and $z\in\End(E_2)$ satisfying~\eqref{eq:xybz}.  Let $I_{z} \subseteq \WW[[t_1,t_2]]$ be the minimal ideal such that there exists an isogeny $\widetilde{z}\in\End_{\WW[[t_1, t_2]]/I_{x,y,b,z}}(\EE_2)$ that reduces to $z$ modulo the maximal ideal of $\WW[[t_1,t_2]]$ and define $I_{u^*}$ similarly where $u^* := b^{\vee}y$.  Since $z,u^*$ are endomorphisms of $E_1$, we can view $I_{z}, I_{u^*}$ as ideals of $\WW[[t_2]]$; similarly we may view $I_{x,u}$ as an ideal of $\WW[[t_1]]$.
		
		\begin{prop}\label{prop:mult}
			The length of $\frac{\WW[[t_1,t_2]]}{I_{x,y,b,z}}$ is bounded above by $2\left(\textup{length }\frac{\WW[[t_1]]}{I_{x,u}}\right)$.  If $\ell \nmid \delta$, then
			\[
				\textup{length }\frac{\WW[[t_1,t_2]]}{I_{x,y,b,z}} =
				\textup{length }\frac{\WW[[t_1]]}{I_{x,u}}
			\]
		\end{prop}
		\begin{proof}
			By the same argument used in Lemma~\ref{lem:RelativelyPrimeIndices} applied to $z, u^*$ instead of $x, u$, either $\Z[z]$ or $\Z[u^*]$ is an order that is maximal at $\ell$.  If $\Z[z]$ is maximal at $\ell$, then define $J := I_{z}$; otherwise define $J := I_{u^*}$.  By definition of $I_{x,y,b,z}, I_{x,u},$ and $J$, we have the containments $I_{x,u}, J \subseteq I_{x,y,b,z}$.  Therefore, we have a surjection
			\[
				\frac{\WW[[t_1,t_2]]}{I_{x,u} + J}
				\twoheadrightarrow
				\frac{\WW[[t_1,t_2]]}{I_{x,y,b,z}}.
			\]
			This gives
			\[
				\textup{length }\frac{\WW[[t_1,t_2]]}{I_{x,y,b,z}} \leq
				\textup{length }\frac{\WW[[t_1,t_2]]}{I_{x,u}+J}.
			\]
			By~\cite{Gross-CanonicalLifts}, $J$ is generated by a linear or quadratic monic polynomial in $t_2$.  Thus
			\[
				\textup{length }\frac{\WW[[t_1,t_2]]}{I_{x,u}+J} \leq
				2\left(\textup{length }\frac{\WW[[t_1]]}{I_{x,u}}\right).
			\]
			This completes the first half of the proof.
			
			Now we assume that $\ell\nmid\delta.$  Since $\deg(b) = \delta$ is prime to $\ell$, $b$ gives an isomorphism between the formal groups of $E_1$ and $E_2$.  Then the argument is exactly the same as in~\cite[Proof of Lemma 5.5]{GrossKeating}.
		\end{proof}

	\subsection{Summary}\label{subsec:summary}
		Now we resume our proof of Theorem~\ref{thm:main}.  Recall that we had shown that
		\[
			\frac{(\CM(K).G_1)_{\ell}}{\log \ell} =
		   \sum_{\substack{\delta\in\Z_{>0}\\D - 4\delta = \square}}
			C_{\delta}\sum_{E_1}\sum_{E_2}
			\sum_{\substack{x,y,b,z\\
			\textup{up to iso.}\\\textup{ as above }}}
			\frac{1}{\#\Aut(x,y,b,z)}  \textup{length }
			\frac{\WW[[t_1, t_2]]}{I_{x,y,b,z}}.
		\]
		The argument in \S\ref{subsec:auto} and Proposition~\ref{prop:mult} show that 
		\[
				\frac{(\CM(K).G_1)_{\ell}}{\log \ell} \leq
				\sum_{\substack{\delta\in\Z_{>0}\\D - 4\delta=\square}}
				C_{\delta}\sum_{E_1}\sum_{E_2}
				\sum_{\substack{x,y,b,z\\
				\textup{up to iso.}\\\textup{ as above }}}
				\frac12  \cdot2\cdot\textup{length }
				\frac{\WW[[t_1]]}{I_{x,yb^{\vee}}},
		\]
		and if $\ell\nmid\delta$, then
		\[
			\frac{(\CM(K).G_1)_{\ell}}{\log \ell} =
			\sum_{\substack{\delta\in\Z_{>0}\\D - 4\delta=\square}}
			C_{\delta}\sum_{E_1}\sum_{E_2}
			\sum_{\substack{x,y,b,z\\
			\textup{up to iso.}\\\textup{ as above }}}
			\frac12  \cdot\textup{length }
			\frac{\WW[[t_1]]}{I_{x,yb^{\vee}}}.
		\]
		Using the results from \S\S\ref{subsec:xu}--\ref{subsec:mult} we will rearrange the terms as follows
		\begin{align*}
			&
				\frac12\sum_{\substack{\delta\in\Z_{>0}\\D - 4\delta=\square}}
				C_{\delta}\sum_{E_1}\sum_{E_2}
				\sum_{\substack{x,y,b,z\\
				\textup{up to iso.}\\\textup{ as above }}}
				\textup{length }
				\frac{\WW[[t_1]]}{I_{x,yb^{\vee}}}\\
				=\;&\frac12
				\sum_{\substack{\delta\in\Z_{>0}\\D - 4\delta=\square}}
				C_{\delta}\sum_{\substack{[(E_1, x, u)]\\\textup{ as above }}}
				 \textup{length }
				\frac{\WW[[t_1]]}{I_{x,u}}\cdot\#\left\{(E_2, y, b, z) \textup{ as above} : u = yb^{\vee}
			   \right\}\\
				=\;&\frac12
				\sum_{\substack{\delta\in\Z_{>0}\\D - 4\delta=\square}}
				C_{\delta}
				\sum_{\substack{n\in\Z \textup{ s.t. }\\
					\frac{\delta^2\Dtilde - n^2}{4D}\in\ell\Z_{>0} \\
					2D|(n + c_K\delta)}}
				\mu_{\ell}(n)
				\sum_{[(E_1, x, u)]\in\calE(n)}
				\#\left\{(E_2, y, b, z)\textup{ as above}:u = yb^{\vee}\right\}\\
				=\;&\frac12
				\sum_{\substack{\delta\in\Z_{>0}\\D - 4\delta=\square}}
				C_{\delta}
				\sum_{\substack{n\in\Z \textup{ s.t. }\\
					\frac{\delta^2\Dtilde - n^2}{4D}\in\ell\Z_{>0} \\
					2D|(n + c_K\delta)}}
				\mu_{\ell}(n)
				\sum_{f_u \in\Z_{>0}}
				\sum_{\substack{[(E_1, x, u)]\in\calE(n)\\
				[\Q(u)\cap\End(E_1) : \Z[u]]
				 = f_u }}
				\#\left\{(E_2, y, b, z)\textup{ as above}:u = yb^{\vee}\right\}\\
				= \;& \frac12\sum_{\substack{\delta\in\Z_{>0}\\D - 4\delta=\square}}
				C_{\delta}
				\sum_{\substack{n\in\Z \textup{ s.t. }\\
				\frac{\delta^2\Dtilde - n^2}{4D}\in\ell\Z_{>0} \\
				2D|(n + c_K\delta)}}{\mu_{\ell}(n)}	
				\sum_{f_u \in\Z_{>0}} \frakI(n,f_u)
				\sum_{\substack{[(E_1, x, u)]\in\calE(n)\\
				[\Q(u)\cap\End(E_1) : \Z[u]]
				 = f_u }}1
				\\
				= \;& \frac12\sum_{\substack{\delta\in\Z_{>0}\\D - 4\delta=\square}}
				C_{\delta}
				\sum_{\substack{n\in\Z \textup{ s.t. }\\
				\frac{\delta^2\Dtilde - n^2}{4D}\in\ell\Z_{>0} \\
				2D|(n + c_K\delta)}}
				{\mu_{\ell}(n)}
				\sum_{f_u\in\Z_{>0}}
				\frakI(n, f_u)\mathscr{J}(d_u(n)f_u^{-2}, d_x(n), t(n, f_u)).
		\end{align*}
		This completes the proof of Theorem~\ref{thm:MainResult}.\qed 
	
\section{Proof of Theorem~\ref{thm:UpperBound}}\label{subsec:ProofOfUpperBound}%
	
	If $\eta$ is any element of $\OO_K\setminus\OO_F$, then given any embedding $\iota\colon\OO_K\hookrightarrow \End(E_1\times E_2)$ we can restrict the domain to obtain an embedding $\iota|_{\OO_F[\eta]}\colon\OO_F[\eta]\hookrightarrow \End(E_1\times E_2)$.  From the definition of $I_{E_1,E_2,\iota}\subseteq \WW[[t_1,t_2]]$, it is clear that
	\[
		\textup{length }\frac{\WW[[t_1,t_2]]}{I_{E_1,E_2,\iota}}\leq
		\textup{length }\frac{\WW[[t_1,t_2]]}{I_{E_1,E_2,\iota|_{\OO_F[\eta]}}}.
	\]
	Since the center of $\End(E_1\times E_2)\otimes\Q$ is exactly $\Q$, it is also clear that
	\[
		\Aut(E_1,E_2,\iota) = \Aut(E_1,E_2,\iota|_{\OO_F[\eta]}).
	\]
	
	If $\iota\colon \OO_F[\eta]\hookrightarrow\End(E_1\times E_2)$ is any embedding ($\iota$ may or may not arise as the restriction of an embedding $\OO_K\hookrightarrow\End(E_1\times E_2)$), then
	\[
		\frac{1}{\#\Aut(E_1,E_2,\iota)}\cdot\textup{length }
		\frac{\WW[[t_1,t_2]]}{I_{E_1,E_2,\iota}}
	\]
	is positive.  Therefore
	\begin{equation}\label{eq:IsectNumber}
		(\CM(K).\textup{G}_1)_{\ell} = \sum_{\substack{E_1,E_2\\
		\iota\colon\OO_K\hookrightarrow\End(E_1\times E_2)}}
		\frac{1}{\#\Aut(E_1,E_2,\iota)}\cdot\textup{length }
		\frac{\WW[[t_1,t_2]]}{I_{E_1,E_2,\iota}}
	\end{equation}
	is bounded above by
	\begin{equation}\label{eq:UpperBound}
		\sum_{\substack{E_1,E_2\\
		\iota\colon\OO_F[\eta]\hookrightarrow\End(E_1\times E_2)}}
		\frac{1}{\#\Aut(E_1,E_2,\iota)}\cdot\textup{length }
		\frac{\WW[[t_1,t_2]]}{I_{E_1,E_2,\iota}}.
	\end{equation}
	
	We compute~\eqref{eq:UpperBound} in the same way that we computed~\eqref{eq:IsectNumber}.  As long as $\eta$ generates an order that is maximal at $\ell$ and all primes $p|\delta$ where $\delta$ is any positive integer such that $D - 4\delta$ is a square, the entire proof goes through verbatim with the exception of Lemma~\ref{lem:RelativelyPrimeIndices}.

	When $\eta$ does not generate the full maximal order $\OO_K$, the arguments in the proof of Lemma~\ref{lem:RelativelyPrimeIndices} prove the following slightly weaker lemma:
	\begin{lemma}\label{lem:FundamentalAtEllNonMaximalCase}
		Let $E$ be a supersingular elliptic curve over $\Fbar_{\ell}$ and let $x, u \in \End(E)$ be endomorphisms satisfying~\eqref{eq:xu}.  Then greatest common divisor of the indices 
		\[
			[\Q(x)\cap\End(E):\Z[x]]\textup{ and }[\Q(u)\cap\End(E):\Z[u]]
		\]
		is supported only at primes dividing $[\OO_K:\OO_F[\eta]]$.  In particular, if $[\OO_K:\OO_F[\eta]]$ is coprime to $\ell$, then at least one of $\Z[x]$, $\Z[u]$ is a quadratic imaginary order maximal at $\ell$.
	\end{lemma}
	As the rest of the proof only requires that at least one of $\Z[x]$ and $\Z[u]$ is maximal at $\ell$ and that any $p|\delta \leq D/4$ does not divide both $[\Q(x)\cap\End(E):\Z[x]]\textup{ and }[\Q(u)\cap\End(E):\Z[u]]$, this lemma, together with our assumption on $\eta$, suffices to complete the proof of Theorem~\ref{thm:UpperBound}.\qed


\section{Embeddings of imaginary quadratic orders into endomorphism rings of 
supersingular elliptic curves}\label{sec:LVSingularModuli}

	In this section, we prove Theorem~\ref{thm:LVSingularModuli} which we restate here for the reader's convenience.
	
	\begin{theorem*}
		Fix $n, f_u\in\Z$ as above, set $d_x := d_x(n), d_u := d_u(n), t := t(n, f_u)$, and write $\OO_u$ for the quadratic imaginary order of discriminant $d_u/f_u^2$.  If the Hilbert symbol 
		\[
		(d_u, D(n^2 - \delta^2\Dtilde))_p = (d_u, (d_uf_u^{-2}d_x - 2t)^2 - d_u{f_u}^{-2}d_x)_p
		\] is equal to $-1$ for some prime $p\ne \ell$, then $\mathscr{J}\left(d_uf_u^{-2}, d_x, t\right) = 0$.  Otherwise $\mathscr{J}\left(d_uf_u^{-2}, d_x, t\right)$ is bounded above by
		\[
			2^{\#\{p\; :\; v_p(t) \geq v_p(d_uf_u^{-2}) > 0, p\nmid 2\ell\}}
			\cdot\tilde\rho_{d_uf_u^{-2}}^{(2)}(t, d_2)\cdot
			\#\left\{\frakb\subseteq\OO_u
			: \Norm(\frakb) = \frac{\delta^2\Dtilde - n^2}{4D\ell f_u^2},
			\frakb \textup{ invertible}\right\},
		\]
		where 
		\[
		\widetilde{\rho}^{(2)}_d(s_0,s_1)  := 
			\left\{
			\begin{array}{ll}
				2 & \textup{if } d \equiv 12 \bmod{16}, 
					s_0\equiv s_1\bmod2\\
				 & \textup{or if } 8\mid d, v_2(s_0) \ge v_2(d) - 2\\
				1& \textup{otherwise}
			\end{array}
			\right\}
			\cdot
			\left\{
			\begin{array}{ll}
				2 & \textup{if } 32\mid d, 4\mid (s_0 - 2s_1)\\
				1& \textup{otherwise}
			\end{array}
			\right\},
		\]
		if $\ell \neq 2$ and $\widetilde{\rho}^{(2)}_d(s_0,s_1) = 1$ if $\ell =2$.
		Furthermore, we have equality in the case that $\frac{\delta^2\Dtilde - n^2}{4Df_u^2}$ is coprime to the conductor of $\OO_u$ and, in all cases, there is an algorithm to compute $\mathscr{J}\left(d_uf_u^{-2}, d_x, t\right)$.
	\end{theorem*}
	
	\subsection{Background}
		The proof of Theorem~\ref{thm:LVSingularModuli} relies heavily on results proved in~\cite{LV-SingularModuli}.  We state the relevant results here and summarize the main ideas of the proofs. The interested reader is referred to~\cite{LV-SingularModuli} for the details.
		
		Let $d_1$ and $d_2$ be discriminants of quadratic imaginary orders and assume that the quadratic imaginary order of discriminant $d_1$ is maximal at $\ell$.  Write $f_i$ for the conductor of the order of discriminant $d_i$.  For every $\SL_2(\Z)$-class of elements in the upper half plane with discriminant $d_1$, we fix a representative $\tau_1$.  Let $E(\tau_1)/\Qbar_{\ell}$ be the elliptic curve with $j$-invariant $j(\tau_1)$. We may assume that $E(\tau_1)$ has good reduction and write $\overline{E(\tau_1)}$ for the reduced elliptic curve over $\Fbar_{\ell}$.  We fix an isomorphism $i_{\tau_1}\colon\Z[\frac12(d_1 + \sqrt{d_1})] \stackrel{\sim}{\to}\End(E(\tau_1))$ and let $\omega_1\in\End(E(\tau_1))$ denote the image of $\frac12(d_1 + \sqrt{d_1})$ in $\End(E(\tau_1))$ under this isomorphism. 
		
		Consider the following set
		\begin{equation}\label{eq:Snm}
			\coprod_{[\tau_1]}\left\{ 
			\begin{array}{rl}
				\phi\in\End(\overline{E(\tau_1)}) : &
				\Tr(\phi) = d_2, \Norm(\phi) = \frac14(d_2^2 - d_2), \\
			&\Tr(\omega_1\cdot\phi^{\vee}) = t, 
			\Q(\phi)\cap\End(\overline{E(\tau_1)}) = \Z[\phi]
			\end{array}\right\}
		\end{equation}
		By~\cite[Thm. 3.1 and proof of Thm. 3.1]{LV-SingularModuli}, we have:
		\begin{theorem}\label{thm:SummaryOfLV}
			Assume that $d_1d_2 \neq (d_1d_2 - 2t)^2$.  If the Hilbert symbol 
			\[
			(d_2, (d_1d_2 - 2t)^2 - d_1d_2)_p 
			\] is equal to $-1$ for some prime $p\ne \ell$, then~\eqref{eq:Snm} is empty.  Otherwise the cardinality of~\eqref{eq:Snm} is bounded above by
			\[
				2^{\#\{p\; :\; v_p(t) \geq v_p(d_1) > 0, p\nmid 2\ell\}}
				\cdot\tilde\rho_{d_1}(t, d_2)\cdot
				\frakA\left(\frac14(d_1d_2 - (d_1d_2 - 2t)^2)\right),
			\]
			where 
			\[
			\widetilde{\rho}^{(2)}_d(s_0,s_1)  := 
				\left\{
				\begin{array}{ll}
					2 & \textup{if } d \equiv 12 \bmod{16}, 
						s_0\equiv s_1\bmod2\\
					 & \textup{or if } 8\mid d, v_2(s_0) \ge v_2(d) - 2\\
					1& \textup{otherwise}
				\end{array}
				\right\}
				\cdot
				\left\{
				\begin{array}{ll}
					2 & \textup{if } 32\mid d, 4\mid (s_0 - 2s_1)\\
					1& \textup{otherwise}
				\end{array}
				\right\}
			\]
            if $\ell \neq 2$ and $\widetilde{\rho}^{(2)}_d(s_0,s_1) = 1$ if $\ell =2$,
			and
				\begin{align*}
					\frakA(N) = &\; \#\left\{
						\begin{array}{ll}
							& \Norm(\frakb) = N, \frakb \textup{ invertible},\\
							\frakb\subseteq\OO_{d_1} : & p\nmid\frakb
							\textup{ for all }p |\textup{gcd}(N, f_2), p\nmid \ell d_1\\
							& \frakp^3\nmid\frakb\textup{ for all }
							\frakp|p|\textup{gcd}(N, f_2, d_1), p\ne\ell
						\end{array}
						\right\}.
				\end{align*}
			Furthermore, this upper bound is an equality in the case that $\frac{d_1d_2 - (d_1d_2 - 2t)^2}{4}$ is coprime to the conductor of $\OO_{d_1}$ and, in all cases, there is an algorithm to compute the cardinality of~\eqref{eq:Snm}.
		\end{theorem}
		\noindent\textit{Idea of proof:}  A calculation shows that the discriminant of the suborder $R := \Z[\omega_1]\oplus\Z[\omega_1]\phi$ is $(\frac14(d_1d_2 - (d_1d_2 - 2t)^2))^2$.  Since, by assumption, this quantity is nonzero, the suborder $R$ has rank $4$ and so must be contained in $\BBl$.  Using arguments like those in Proposition~\ref{prop:ell-divisibility}, one shows that $d_1d_2 > (d_1d_2 - 2t)^2$ and thus we obtain the Hilbert symbol statement.  
		
		To prove the upper bound, we need to develop more machinery.  In~\cite[\S6]{LV-SingularModuli}, we give explicit presentations of $\End(\overline{E(\tau_1)})$ as suborders of $\Mat_2(\Q(\sqrt{d_1})).$  Using this presentation, one shows that elements $\phi$ of fixed norm and trace give rise to invertible ideals in $\OO_{d_1}$ that have a fixed ideal class in $\frac{\Pic\OO_{d_1}}{2\Pic\OO_{d_1}}.$  Moreover, multiple elements can give rise to the same ideal only if $t$ is sufficiently divisible by primes dividing $d_1$.  
		
		If $\frac14(d_1d_2 - (d_1d_2 - 2t)^2)$ is coprime to the conductor of $\OO_{d_1}$, then the converse holds, i.e., given an ideal in a fixed ideal class, one can construct one (or multiple, depending on $t$) endomorphisms $\phi$ with the desired properties.  The interested reader can find the details in~\cite[\S \S 5,6]{LV-SingularModuli}.
		
	\subsection{Proof of Theorem~\ref{thm:LVSingularModuli}}
	
		Let $d_1$ and $d_2$ be discriminants of quadratic imaginary orders and assume that the quadratic imaginary order of discriminant $d_1$ is maximal at $\ell$.  Recall that $\mathscr{J}(d_1,d_2,t)$ equals
		\[
			\sum_{E/\Fbar_{\ell}}\#\left\{
			\begin{array}{ll}
				i_j\colon\Z\left[\frac{d_j + \sqrt{d_j}}2\right]\hookrightarrow 
				\End(E) : &
				\Tr(i_1(d_1 + \sqrt{d_1})i_2(d_2 - \sqrt{d_2})) = 4t,\\
				& i_1(\Q(\sqrt{d_1}))\cap\End(E) = i_1(\Z\left[\frac{d_1 + \sqrt{d_1}}2\right])
			\end{array}
			\right\}/\End(E)^{\times}.
		\]
		We will relate $\mathscr{J}(d_1,d_2,t)$ to the number of endomorphisms of reductions of elliptic curves with complex multiplication; precisely, we will show that $\mathscr{J}(d_1,d_2,t)$ equals
		\[
			\frac{4}{w_1e}\sum_{[\tau_1]}\#\left\{ 
			\phi\in\End(\overline{E(\tau_1)}) : 
			\Tr(\phi) = d_2, \Norm(\phi) = \frac14(d_2^2 - d_2),
			\Tr(i_{\tau_1}(d_1 + \sqrt{d_1})\cdot\phi^{\vee}) = 2t\right\}.
		\]
	
		Let $E/\Fbar_{\ell}$ be an elliptic curve and let $i_1\colon \Z\left[\frac{d_1 + \sqrt{d_1}}2\right]\hookrightarrow \End(E)$ be an embedding such that $i_1(\Q(\sqrt{d_1}))\cap\End(E) = i_1(\Z\left[\frac{d_1 + \sqrt{d_1}}2\right])$.  By Deuring's lifting theorem\cite[Chap. 13, Thm. 14]{Lang-Elliptic}, there exists a $\tau_1$ in the upper half-plane of discriminant $d_1$ such that $\overline{E(\tau_1)}$ is isomorphic to $E$.  Furthermore, after possibly replacing $E$ with an isomorphic curve, and conjugating $i_1,i_2$ by an automorphism $\psi$ of $E$, we may assume that the embedding $i_{\tau_1}\colon \Z\left[\frac{d_1 + \sqrt{d_1}}{2}\right]\stackrel{\sim}{\to}\End(E(\tau_1)) \hookrightarrow\End(\overline{E(\tau_1)})$ either agrees with $i_1$ or differs from $i_1$ by precomposition with the nontrivial Galois automorphism.  By~\cite{Gross-CanonicalLifts}, the class of $\tau_1$ modulo $\SL_2(\Z)$ is unique if $\ell\nmid d_1$ and otherwise there are exactly two choices for the class of $\tau_1$.  Moreover, the choice of $\psi/\{\pm 1\}$ is unique up to multiplication by units in $(\im i_1)/\{\pm 1\}$.  
		
		Conversely, every $\tau_1$ gives rise to an elliptic curve $\overline{E(\tau_1)}/\Fbar_{\ell}$ and an embedding 
		\[
			i_{\tau_1}\colon \Z\left[\frac{d_1 + \sqrt{d_1}}{2}\right]
			\stackrel{\sim}{\to}
			\End(E(\tau_1)) \hookrightarrow\End(\overline{E(\tau_1)}).
		\]
		By~\cite[Prop. 2.2]{LV-SingularModuli}, we have $i_{\tau_1}(\Q(\sqrt{d_1}))\cap\End(E) = i_{\tau_1}(\Z\left[\frac{d_1 + \sqrt{d_1}}2\right])$.  Thus,
		\[
			\sum_{E/\Fbar_{\ell}}\#\frac{\{i_1\colon \Z\left[\frac{d_1 + \sqrt{d_1}}2\right]\hookrightarrow \End(E)\}}{\End(E)^{\times}} = \frac2e\cdot\#\{[\tau_1]: \textup{disc}(\tau_1) = d_1\},
		\]
		where $e$ denotes the ramification index of $\ell$ in $\Q(\sqrt{d_1})$.
		
		Now fix an element $\tau_1$ and fix an embedding $i_2\colon\Z\left[\frac{d_2 + \sqrt{d_2}}{2}\right]\hookrightarrow\End(\overline{E(\tau_1)})$ such that $\Tr(i_{\tau_1}(d_1 + \sqrt{d_1})i_2(d_2 - \sqrt{d_2})) = 4t$.  Then $i_2$ uniquely determines an element $\phi\in\End(\overline{E(\tau_1))})$ such that
		\[
			\Tr(\phi) = d_2, \quad \Norm(\phi) = \frac14(d_2^2 - d_2),
			\quad \Tr(i_{\tau_1}(d_1 + \sqrt{d_1})\phi^{\vee}) = 2t,
		\]
		namely $\phi := i_2\left(\frac{d_2 + \sqrt{d_2}}2\right)$.  Conversely, a choice of $\phi$ uniquely determines an embedding $i_2\colon \Z\left[\frac{d_2 + \sqrt{d_2}}2\right]\hookrightarrow(\End(\overline{\tau_1}))$.  Therefore, $\mathscr{J}(d_1,d_2,t)$ equals
				\[
					\frac{1}{e}\cdot \frac{2}{w_1}\cdot2\sum_{[\tau_1]}\#\left\{ 
					\phi\in\End(\overline{E(\tau_1)}) : 
					\Tr(\phi) = d_2, \Norm(\phi) = \frac14(d_2^2 - d_2),
					\Tr(i_{\tau_1}(d_1 + \sqrt{d_1})\cdot\phi^{\vee}) = 2t\right\}.
				\]
				
			In~\cite[Thm. 3.1]{LV-SingularModuli}, the present authors explain how to compute
			\[
				\sum_{[\tau_1]}\#\left\{
				\begin{array}{ll} 
					\phi\in\End(\overline{E(\tau_1)}) : &
						\Tr(\phi) = d_2, \Norm(\phi) = \frac14(d_2^2 - d_2),\\
						&\Tr(i_{\tau_1}(d_1 + \sqrt{d_1})\cdot\phi^{\vee}) = 2t,
						\Q(\phi)\cap\End(E(\tau_1)) = \Z[\phi]
				\end{array}\right\}.
			\]
			It is straightforward to see how to modify the proof of~\cite[Thm. 3.1]{LV-SingularModuli} in order to omit the last condition, that is, the condition that $\Q(\phi) \cap\End(E(\tau_1)) = \Z[\phi]$.  Roughly speaking, one should omit every step that involves the conductor of the order of discriminant $d_2$, as only the condition that $\Q(\phi) \cap\End(E(\tau_1)) = \Z[\phi]$ depends on this conductor.  After making these changes to the proof, one proves that the quantity
			\begin{equation*}
				\sum_{[\tau_1]}\#\left\{ 
					\phi\in\End(\overline{E(\tau_1)}) : 
					\Tr(\phi) = d_2, \Norm(\phi) = \frac14(d_2^2 - d_2),
					\Tr(i_{\tau_1}(d_1 + \sqrt{d_1})\cdot\phi^{\vee}) = 2t\right\}
			\end{equation*}
			is $0$ if there exists a prime $p\ne \ell$ such that the Hilbert symbol
			\[
				(d_u, D(n^2 - \delta^2\Dtilde))_p = 
				(d_u, (d_uf_u^{-2}d_x - 2t)^2 - d_u{f_u}^{-2}d_x)_p = -1,
			\]
			and otherwise, that it is bounded above by
			\[
			2^{\#\{p\; :\; v_p(t) \geq v_p(d_uf_u^{-2}) > 0, p\nmid 2\ell\}}
			\cdot\tilde\rho_{d_uf_u^{-2}}^{(2)}(t, d_2)\cdot
			\#\left\{\frakb\subseteq\OO_u
			: \Norm(\frakb) = \frac{\delta^2\Dtilde - n^2}{4D\ell f_u^2},
			\frakb \textup{ invertible}\right\}.
			\]
			One also shows that the upper bound is an equality in the case that $\frac{\delta^2\Dtilde - n^2}{4Df_u^2}$ is relatively prime to the conductor of the order of discriminant $d_uf_u^{-2}$.  This should not be surprising, as it is basically the statement of Theorem~\ref{thm:SummaryOfLV} with the conditions involving $f_2$, the conductor of the order of discriminant $d_2$, omitted. 
			\qed
		
			\begin{remark}
				There is an alternative way of proving Theorem~\ref{thm:LVSingularModuli} that does not require making the necessary modifications to the proof of~\cite[Thm. 3.1]{LV-SingularModuli}.  First one  notes that
				\[
				\left\{ 
				\begin{array}{rl}
					\phi\in\End(\overline{E(\tau_1)}) : &
					\Tr(\phi) = d_2, \Norm(\phi) = \frac14(d_2^2 - d_2), \\
				&\Tr(\omega_1\cdot\phi^{\vee}) = t, 
				\end{array}\right\}
				\]
				equals
				\[
				\coprod_{f|f_2}\left\{ 
				\begin{array}{rl}
					 &
					\Tr(\tilde\phi) = d_2f^{-2}, 
					\Norm(\tilde\phi) = \frac14(d_2^2f^{-4} - d_2f^{-2}), \\
					\tilde\phi\in\End(\overline{E(\tau_1)}) :
					&\Tr(\omega_1\cdot\tilde\phi^{\vee}) = 
					\frac1{2f^2}(2ft + d_1f - d_2f + d_2),\\ 
					&\Q(\tilde\phi)\cap\End(\overline{E(\tau_1)}) 
					= \Z[\tilde\phi]
				\end{array}\right\},
				\]
				where $f$ ranges over all positive divisors of $f_2$, the conductor of the order of discriminant $d_2$; the map $\tilde\phi \mapsto \frac12(2f\tilde\phi - d_2f^{-1} + d_2)$ gives a bijective map from the latter set to the former.  Then, one uses repeated applications of Theorem~\ref{thm:SummaryOfLV} to compute the cardinality of the latter set.  A series of algebraic manipulations will complete the proof.
			\end{remark}

\section{Ideals in $\BBl$}\label{sec:IdealsInM2Zp}

	In this section we prove Theorem~\ref{thm:CountingIdeals}, which we restate here for the reader's convenience.  Recall that for any integral ideal $I$ in $\BBl$, $\RO(I) = \{y\in\BBl : Iy\subseteq I\}$ is the right order of $I$.
	
	\begin{theorem*}
		Fix $R$ a maximal order in $\BBl$.  Assume that $x,u\in R$ and $\gamma,\delta\in \Z$ are such that
		\[
			\Tr(u), \quad \Norm(u), \quad \textup{and }\quad
			\Tr(xu^{\vee}) + \gamma\Norm(u)/\delta \quad 
			\textup{ are }0 \textup{ modulo }\delta.
		\]
		Define $w := x + \gamma u/\delta$, $c_p\in\Z$ to be such that $up^{-c_p}\in R_p\setminus pR_p$, and $r_p := \max(v_p(\delta) - c_p, 0)$.  Assume that for all $p|\delta, p\ne\ell$, either $c_p = 0$ or $\Q_p(p^{r_p}w)\cap R\otimes\Z_p = \Z_p[p^rw]$.

		Then $\#\left\{I\subseteq R : \delta,u\in I, \Norm(I) = \delta, 
		\textup{ and } w \in \RO(I)\right\}$ equals
		\[
			\prod_{p|\delta, p\ne\ell}\left(
			\sum_{\substack{j = 0 \\j \equiv v_p(\delta)\pmod 2}}^{v_p(\delta)} 
			\frakI^{(p)}_{j - r_p}(\Tr(w), \Norm(w))\right),
		\]
		where 
		\[
			\frakI_{C}^{(p)}(a_1, a_0) = 
			\begin{cases}
				\#\{\widetilde{t} \bmod p^{C} : 
					\widetilde{t}^2 - a_1\widetilde{t} + a_0
					\equiv 0 \pmod{p^{C}}\}
				& \textup{ if }  C\geq 0,\\
				0 & \textup{ if } C<0.
			\end{cases}
		\]
	\end{theorem*}
	
	This section is independent of the rest of the paper, so we disregard any notation fixed elsewhere.  
	
	For any prime $p$, let $R_p := R\otimes_{\Z}\Z_p$.  If $p\neq \ell$, then after fixing an isomorphism of $\BBl\otimes\Q_{p}$ with $\Mat_2(\Q_p)$ we can view $R_p$ as a maximal order in $\Mat_2(\Q_p)$. Moreover, after conjugating by an appropriate element, we may assume that $R_p = \Mat_2(\Z_p)$.  If $I_p$ is an ideal in $\Mat_2(\Z_p)$, then $\RO(I_p) := \{A\in\Mat_2(\Q_p): I_pA\subseteq I_p\}$.  By~\cite{Vigneras}*{Chap. 2, Thm. 2.3}, there are $1 + p + \cdots + p^N$ ideals of norm $p^N$ in $\Mat_2(\Z_p)$, and they are all of the form
	\begin{equation}\label{eq:localideal}
		R_p\begin{pmatrix}
			p^n & t\\
			0 & p^m
		\end{pmatrix},
	\end{equation}
	where $n,m$ are positive integers such that $n + m = N$, and $t\in \Z_p$.  The triple $(n, m, t\bmod{p^m})$ uniquely determines the ideal.  By abuse of notation, we will use the triple $(n,m,t)$ to refer to both the element $\begin{pmatrix} p^n & t\\0 & p^m \end{pmatrix}$ and the ideal it generates.

	We say an element $y\in \Mat_2(\Q_p)$ is \textit{optimally embedded} \bvchange{in $R_p$} if $\Q_p(y)\cap R_p = \Z_p[y]$ and that $y$ is \textit{primitive} if $y\in R_p\setminus pR_p$.  An ideal $I$ is \textit{primitive} if it is generated by a primitive element, i.e. if $I = (n, m, t)$ where at least one of $n, m, $ or $v(t)$ are zero.  We divide primitive ideals into three cases: Case 1) $n = 0$, Case 2) $m = 0$, and Case 3) $n,m >0$, and $v(t) = 0$.  Note that these cases are mutually exclusive unless we are considering the unit ideal.
	
	In~\S\ref{sec:RightOrders} we give a formula that computes, for a fixed integral element $y\in \Mat_2(\Q_p)$ and integer $N$, the number of ideals $I_p$ of norm $p^N$ with $y\in\RO(I_p)$.  In~\S\ref{sec:IdealContainment}, we give a criterion to determine whether one ideal is contained in another.  In~\S\ref{sec:SummaryOfCount}, we explain how the results in the two previous sections come together to prove Theorem~\ref{thm:CountingIdeals}.
	
	\subsection{Right orders of ideals in $\Mat_2(\Z_p)$}\label{sec:RightOrders}
		
		\begin{lemma}\label{lem:ratlcanform}
			Let $y\in\Mat_2(\Q_p)$ be an integral element.  Assume that there exists an $r\in\Z_{\geq0}$ such that $p^ry$ is optimally embedded.	Then there exists an $A \in \GL_2(\Z_p)$ such that
			\[
				AyA^{-1} = 
				\begin{pmatrix}0 & -\Norm(y)p^r\\p^{-r} & \Tr(y)\end{pmatrix}
			\]
		\end{lemma}
		\begin{proof}
			Write $y = \begin{pmatrix}a&b\\c&d\end{pmatrix}.$  \bvchange{Since $p^ry$ is optimally embedded in $R_p = \Mat_2(\Z_p)$, the element
            \[
                p^ry - p^ra = 
                \begin{pmatrix}
                0&p^rb\\p^rc&p^r(d-a)\end{pmatrix}    
            \]
            is primitive.  Therefore, at least one of $b$, $c$, or $a-d$ has valuation exactly $-r$.}
            
			If $v(b) = -r$, then let $A = \begin{pmatrix} -dp^r & bp^r\\1 & 0\end{pmatrix}.$  If $v(c) = -r$, then let $A = \begin{pmatrix} cp^r & -ap^r\\0&1\end{pmatrix}$.  If $v(b), v(c) > -r,$ and $v(a - d) = -r$, then let $A = \begin{pmatrix} (c - d)p^r & (b - a)p^r\\1&1\end{pmatrix}.$  One can easily check that these matrices fulfill the assertions in the lemma.
		\end{proof}

		\begin{prop}\label{prop:rightorders}
			Let $y$ be an integral element of $\Mat_2(\Q_p)$. Assume that there exists an $r\in\Z_{\geq0}$ such that $p^ry$ is optimally embedded.  Then the number of primitive ideals $I_p$ of norm $p^N$ such that $y\in \RO(I_p)$ is $\frakI^{(p)}_{N-r}(\Tr(y),\Norm(y))$, where: 
			\[
				\frakI_{\bvchange{m}}^{(p)}(a_1, a_0) = 
					\begin{cases}
						\#\{\widetilde{t} \bmod p^{N-r} : 
							\widetilde{t}^2 - a_1\widetilde{t} + a_0
							\equiv 0 \pmod{p^{\bvchange{m}}}\}
						& \textup{ if }  \bvchange{m\geq0},\\
						0 & \textup{ if } \bvchange{m<0}.
					\end{cases}
			\]
			In particular, there is a unique primitive ideal $J_p$ of norm $p^r$ such that $y\in \RO(J_p)$.  Furthermore, if $I_p$ is any other ideal \bvchange{(not necessarily primitive)} such that $y\in\RO(I_p)$ then $I_p\subseteq J_p$.
		\end{prop}
		\begin{proof}
			By Lemma~\ref{lem:ratlcanform}, there exists $A\in \GL_2(\Z_p)$ such that $\tilde y := AyA^{-1} =  \begin{pmatrix} 0 & -\Norm(y)p^{r} \\ p^{-r} & \Tr(y) \end{pmatrix}$.  		
			Recall that an element $y$ is in the right order of $R_pT$ if and only if $TyT^{-1} \in R_p$.  Therefore, $y$ is in the right order of $R_pT$ if and only if $\tilde y$ is in the right order of $R_pTA^{-1}$.Thus it suffices to count the number of ideals $I_p$ such that $\tilde y \in \RO(I_p)$.  
	
			Let $I_p$ correspond to the triple $(n,m,t)$.  Then $\tilde y\in \RO(I_p)$ if and only if 
			\begin{equation}
				m - n - r \geq 0, \quad v(t) - n - r \geq 0, \quad t^2p^{-r - N} - t\Tr(y)p^{-m} + \Norm(y)p^{n - m + r} \in \Z_p.
			\end{equation}
			
			The first two conditions imply that $m \geq n$ and $v(t) \geq n$.  Since one of $m,n, v(t)$ must be $0$, this implies that $n = 0$ and $m = N$.  Now the first condition shows that there are no solutions if $N < r$; so from now on we assume that $r \leq N$.  The second condition implies that $t = p^r\widetilde{t}$; substituting this into the third condition we obtain
			\begin{equation}\label{eq:tCondition}
				\widetilde{t}^2p^{r - N} - \Tr(y)\widetilde{t}p^{r - N} + \Norm(y)p^{r - N} \in \Z_p.
			\end{equation}
			This completes the proof of the formula for $\frakI_{M-r}(\Tr(y),\Norm(y))$.
			
			The argument above shows that any ideal $I_p$ such that $y\in \RO(I_p)$ is equal to
			\[
				R_p\bvchange{p^k}\begin{pmatrix} 1 &\widetilde{t}p^r\\0& p^N\end{pmatrix}A,
			\]
			where $\widetilde{t}$ satisfies~\eqref{eq:tCondition}, $N\geq r$, \bvchange{and $k$ is any nonnegative integer}; in particular, if $N = r$, there is a unique ideal $J_p:= R_p\begin{pmatrix} 1& 0 \\0& p^r\end{pmatrix}A$ such that $y\in \RO(J_p)$.  Since
			\[
				\begin{pmatrix} 1 &\widetilde{t}p^r\\0& p^N\end{pmatrix}A\cdot
				\left[\begin{pmatrix} 1& 0 \\0& p^r\end{pmatrix}A\right]^{-1}
				=
				\begin{pmatrix}1&\widetilde{t}\\0& p^{N -r}\end{pmatrix}\in R_p,
			\]
			$I_p\subseteq J_p$, as desired.
		\end{proof}
		
	\subsection{Lattice of ideals in $\Mat_2(\Z_p)$}\label{sec:IdealContainment}
		
		\begin{lemma}\label{lem:unique-primitive}
			Let $z$ be a primitive integral element of $\Mat_2(\Z_p)$.  Then there is a unique ideal of norm $p^N$ containing $z$ for all $N \leq v_p(\Norm(z))$.  We write $I_{z, N}$ for this unique ideal.
		\end{lemma}

		\begin{proof}
			Let $(n,m,t)$ be a generator for the ideal $R_pz$, i.e. $(n,m,t) = \epsilon z$, for some $\epsilon \in R_p^{\times}$.  Then $(n,m,t)$ is contained in an ideal $I$ if and only if $z$ is contained in $I$.  Assume that $(n,m,t)$ is contained in $(n', m', t')$, where $m' + n' = N$.  Thus, the product
			\[
				\begin{pmatrix}
					p^n & t\\
					0 & p^m
				\end{pmatrix}
				\begin{pmatrix}
					p^{-n'} & -t'p^{-N}\\
					0 & p^{-m'}
				\end{pmatrix}
				= \begin{pmatrix}
					p^{n-n'} & tp^{-m'} - t'p^{n - N}\\
					0 & p^{m-m'}
				\end{pmatrix}
			\]
			must be in $R_p$.  Therefore, $n \geq n', m \geq m', t \equiv t' p^{n - n'} \pmod{p^{m'}}$.  A case-by-case analysis shows that there is a unique primitive tuple $(n', m', t')$ with $n' + m' = N$ that satisfies these conditions; they are listed here for the readers' convenience.
			\begin{align}
				n = 0 & \Rightarrow n' = 0, m' = N, t' \equiv t \pmod{p^{N}},\label{eq:containment1}\\
				m = 0 & \Rightarrow n' = N, m' = 0, \label{eq:containment2}\\
				v(t) = 0 & \Rightarrow n' = \min(n, N), m' = N - n',  
				t' \equiv t \pmod {p^{m'}}. \label{eq:containment3}
			\end{align}
			We remark that there is no condition on $t'$ in~\eqref{eq:containment2} since $t'$ is only defined modulo $p^{m'} = 1.$
		\end{proof}

		\begin{lemma}\label{lem:scalar-containment}
			Let $j,k,r,s$ be non-negative integers and let $y,z$ be primitive elements of norm at least $p^r, p^s$ respectively.  Then $p^j I_{y,r} \subseteq p^kI_{z,s}$ if and only if $r,s \geq s - j + k$ and $I_{y,s - j + k} = I_{z, s - j + k}$.  (If $s - j + k < 0$ then this last condition is vacuous.)
		\end{lemma}
		\begin{proof}
			We prove the backwards direction first.  Let $y_N, z_N$ denote generators for $I_{y,N}$, $I_{z,N}$ respectively for any (valid) integer $N$.  Since  $s \geq s - j + k$ and $I_{y, s- j + k} = I_{z, s - j + k}$, we may write $z_s$ as $z'y_{s - j + k}$ for some $z'\in R_p$ of norm $p^{j - k}$.  We rewrite $p^j y_r(p^kz_s)^{-1}$ as follows
			\[
				p^{j-k}y_r(z' y_{s - j + k})^{-1} = y_ry_{s - j + k}^{-1}\cdot p^{j - k}z'^{-1}.
			\]
			By definition of $y_N$ and since $r\geq s - j + k$, $y_ry_{s - j + k}^{-1}\in R_p$.  Additionally, since $z'$ has norm $p^{j - k}$, $p^{j - k}z'^{-1}\in R_p$.  Thus $p^jI_{y,r}\subseteq p^kI_{z,s}$.
			
			Now we consider the forward direction; assume that $p^{j}I_{y,r}\subseteq p^k I_{z,s}$.  Then $p^{j - k}I_{y,r}\subseteq I_{z,s}\subseteq R_p$.  Since $I_{y,r}$ is primitive, this implies that $j \geq  k$, or equivalently that $s\geq s + k - j$.  Without loss of generality we reduce to the case that $k = 0$.  
			
			If $s \leq j$, then all remaining conditions are vacuous, so we assume that $s > j$.  Let $(n_r, m_r, t_r)$ be a generator of $I_{y,r}$ and $(n_s, m_s, t_s)$ be a generator of $I_{z,s}$.  By assumption, we have
			\[
				p^j\begin{pmatrix}p^{n_r} & t_r \\ 0& p^{m_r}\end{pmatrix}
				\begin{pmatrix}p^{-n_s} & -t_sp^{-s} \\ 0& p^{-m_s}\end{pmatrix}
				= \begin{pmatrix}p^{j + n_r - n_s} & 
				t_rp^{j - m_s} - t_sp^{j + n_r - s}\\
				0 & p^{j + m_r - m_s}\end{pmatrix}\in R_p,
			\]
			or, equivalently
			\[
				j + n_r \geq n_s, \quad j + m_r \geq m_s, \quad 
				t_rp^{n_s} \equiv t_sp^{n_r}\pmod{p^{s - j}}.
			\]
			If $n _s = 0$, then $j + r \geq j + m_r\geq m_s = s$.  Similarly if $m_s = 0$, then $j + r \geq j + n_r \geq n_s = s$.  If $n_s,m_s >0$, then $v(t_s) = 0$.  Since $t_rt_s^{-1}p^{j + m_r - m_s} - p^{j + r - s}\in \Z_p$, this again implies that $j + r \geq s$.  It remains to show that $I_{z, s- j} = I_{y, s-j}$.
			
			First we treat the case when $n_r = 0$.  Then $m_s >0$ and $t_s\equiv t_rp^{n_s}$.  Since at least one of $v(t_s)$, $n_s$ must be zero, this shows that $n_s = 0$ and $t_r \equiv t_s\pmod p^{s - j}$.  Using~\eqref{eq:containment1}--\eqref{eq:containment3}, we see that $I_{y, s-j} \leftrightarrow (0, s-j, t_r\bmod p^{s - j})$ and $I_{z, s-j} \leftrightarrow (0, s-j, t_s\bmod p^{s-j})$ so we have equality.
			
			Now consider the case when $m_r = 0$.  Then by~\eqref{eq:containment1}--\eqref{eq:containment3} $I_{y, s-j} \leftrightarrow (s-j, 0, 0)$.  Since $j \geq m_s$, $n_s\geq s - j >0 $.  By another application of~\eqref{eq:containment1}--\eqref{eq:containment3}, we see that regardless of whether $m_s = 0 $ or $m_s>0$ and $v(t_s) = 0$ we have $I_{z, s-j} \leftrightarrow (s-j, 0, 0)$.
			
			Finally we consider the case where $n_r, m_r > 0$ and $v(t_r) = 0$.  If $n_r < s - j$, then $I_{y, s-j} \leftrightarrow (n_r, s- j - n_r, t_r\bmod p^{s - j - n_r})$.  In this case, the conditions above show that $m_s, n_s >0 $.  This in turn implies that $t_s$ is a $p$-adic unit, and since $t_sp^{n_r}\equiv t_rp^{n_s}\pmod p^{s-j}$ we have $n_s = n_r < s_j$.  Then, by~\eqref{eq:containment1}--\eqref{eq:containment3}, $I_{z, s-j}\leftrightarrow (n_r, s- j - n_r, t_s\bmod p^{s - j - n_r})$, which is equal to $I_{y, s-j}$.  The sole remaining case is when $n_r\geq s - j$ which implies that $I_{y, s-j} = (s - j, 0, 0)$.  Since $t_st_r^{-1}p^{n_r}\equiv p^{n_s}\pmod{p^{s- j}}$, $n_s \geq s-j$.  As in the previous paragraph, this means that regardless of whether $m_s = 0 $ or $m_s>0$ and $v(t_s) = 0$ we have $I_{z, s-j} \leftrightarrow (s-j, 0, 0)$.
		\end{proof}

	\subsection{Proof of Theorem~\ref{thm:CountingIdeals}}
	\label{sec:SummaryOfCount}

		Recall that $R$ is a fixed maximal order in $\BBl$ and $x,u\in R$ and $\gamma,\delta\in \Z$ are such that
		\[
			\Tr(u), \quad \Norm(u), \quad \textup{and }\quad
			\Tr(xu^{\vee}) + \gamma\Norm(u)/\delta \quad 
			\textup{ are }0 \textup{ modulo }\delta.
		\]
		We are interested in computing the number of left integral ideals $I$ of $R$ that satisfy
		\begin{equation}\label{eq:IdealConditions}
			\delta,u\in I, \quad \Norm(I) = \delta, \quad \textup{ and }
			w := x + \gamma u/\delta \in \RO(I),
		\end{equation}
		where $\RO(I) = \{y\in\BBl : Iy\subseteq I\}$ is the right order of $I$.  Note that, due to the assumptions above, $w$ is {integral}, i.e. $\Norm(w)$ and $\Tr(w)$ are in $\Z$.

		For any prime $p$, let $R_p := R\otimes_{\Z}\Z_p$.  By~\cite{Vigneras}*{Chap. 3, Prop. 5.1}, the map
		\[
			\left\{\textup{left ideals of }R\right\}\to
			{\prod_{p}}'\left\{\textup{left ideals of }R_p\right\}, \quad 
			I \mapsto \left(I_p\right)
		\]
		is a bijection ($I_p := I\otimes_{\Z}\Z_p$).  Thus
		\[
			\#\{ I\subseteq R : 
			I \textup{ satisfies }~\eqref{eq:IdealConditions}\} 
			= \prod_{p}
			\#\{ I_p\subseteq R_p : 
				I_p \textup{ satisfies }~\eqref{eq:IdealConditions}\}
		\]
		If $p\nmid\delta$, then the first condition of~\eqref{eq:IdealConditions} implies that $\Norm(I_p) = \langle1\rangle$ and so $I_p = R_p$.  If $p = \ell$, then $R_{\ell}$ is the unique maximal order in $\BBl\otimes_{\Q}\Q_{\ell}$ and ideals in $R_{\ell}$ are completely classified by the $\ell$-valuation of their norms, and for any ideal $I_{\ell}\subseteq R_{\ell}$ we have that $R(I_{\ell}) = R_{\ell}$.  Since $w$ is integral and $\delta\mid\Norm(u)$ it is clear that the ideal of norm $\ell^{v(\delta)}$ satisfies conditions~\eqref{eq:IdealConditions}.  Thus for all $p$ outside the finite set $\{ p : p | \delta, p \ne \ell\}$, we have
		\[
			\#\left\{I_p \subseteq R_p := R\otimes_{\Z}\Z_p :
			\textup{ satisfying }~\eqref{eq:IdealConditions}\right\} = 1.
		\]
		
		Henceforth we assume that $p|\delta$ and $p\ne\ell$.  Recall that $c_p\in\Z$ is such that $up^{-c_p}\in R_p\setminus pR_p$ and $r_p = \max(v_p(\delta) - c_p, 0)$.

		\begin{lemma}\label{lem:RuRdelta}
			We have $w\in\RO(R_pu + R_p\delta)$,  \bvchange{ $p^{2c_p + r_p}|\Norm(u)$}, and the norm of $R_pu + R_p\delta$ divides $\delta^2p^{-r_p}$.
		\end{lemma}
		\begin{proof}
			In order to prove that $w \in\RO(R_pu + R_p\delta)$, we will show that $\delta w$ and $uw$ are both contained in $R_pu + R_p\delta$.  The first containment is straightforward.  For the second containment, we need the fact that $\Tr(ab) = \Tr(ba)$ for any $a,b\in R_p$ and~\eqref{eq:xu}.  Consider the following expansion
			\begin{align*}
				uw = & \Tr(u)w - u^{\vee}w = 
					\Tr(u) w - u^{\vee}x - {\gamma u^{\vee}u}/{\delta}
				    \\
					= & \Tr(u) w + x^{\vee}u - 
					(\Tr(u^{\vee}x) + {\gamma\Norm(u)}/{\delta}).
			\end{align*}
			Since, by assumption, $\Tr(u)$ and $\Tr(xu^{\vee} + \gamma\Norm(u)/\delta$ are divisible by $\delta$, $uw\in R_pu + R_p\delta.$
			
            \bvchange{Now we show that $v_p(\Norm(u))\geq 2c_p + r_p$.  If $c_p = 0$, then this follows from our assumptions on $x,u,\delta,$ and $\gamma$.  If $r_p=0$ then this follows from the definition of $c_p$.  
			
			Assume that $c_p,r_p >0$ and that $v_p(\Norm(u))< 2c_p + r_p = c_p + v_p(\delta)$. Using the criterion in Lemma~\ref{lem:scalar-containment} we can show that $\delta\in R_pu$ so $R_pu + R_p\delta = R_p u$.  Since $\RO(R_pu) = \RO(R_pup^{-c_p})$, by the first part of the proof $w$ is in the right order of an ideal of norm $\Norm(u) - 2c_p < r_p$.  However, Proposition~\ref{prop:rightorders} shows that this is impossible since by assumption $p^{r_p}w$ is optimally embedded in $\Mat_2(\Z_p)$.  This proves the claim.}
            
			Now we compute the norm of $R_pu + R_p\delta$.  If $r_p = 0$, then $u\in R_p\delta$ and $\Norm(R_pu + R_p\delta) = \Norm(R_p\delta) = \delta^2.$  Now assume that $r_p >0$, so $v_p(\Norm(u))\geq c_p + v_p(\delta)$.  
 %
            %
            If $\Norm(u) = c_p + v_p(\delta)$ then clearly $\Norm(R_p u + R_p\delta)|\delta p^{c_p}$.  Assume that $\Norm(u) > c_p + v_p(\delta)$ and write $u$ as $\begin{pmatrix}a_{1,1} & a_{1,2}\\a_{2,1} & a_{2,2}\end{pmatrix}$.  By the definition of $c_p$, there exists $i,j$ such that $v_p(a_{i,j}) = c_p$.  Define $A\in\Mat_2(\Z_p)$ to be such that row $i$ and column $j$ consist only of zeros and the remaining entry has a $1$.  Then 
			\[
				\Norm(u + \delta A) = \Norm(u) + \delta\Tr(Au^{\vee}) = \Norm(u) + \delta a_{i,j}.
			\]
			Since $u + \delta A\in R_pu + R_p\delta$, this shows that $\Norm(R_pu + R_p\delta) | \delta p^{c_p}$, which completes the proof.
		\end{proof}

		Now we are in a position to prove that there are
		\[
			\sum_{\substack{j = 0 \\j \equiv v_p(\delta)\pmod 2}}^{v_p(\delta)} \frakI^{(p)}_{j - r_p}(\Tr(w), \Norm(w))
		\]
		many ideals $I_p \subset R_p$ that satisfy~\eqref{eq:IdealConditions}.
		
		If $c_p = 0$, then the sum is $1$, so we must prove that there is a unique ideal that satisfies~\eqref{eq:IdealConditions}.  In this case $u$ is a primitive element of $R_p$ so Lemmas~\ref{lem:unique-primitive} and~\ref{lem:scalar-containment} imply that there is a unique ideal of norm $\delta$ that contains $u$, $I_{u, v(\delta)}$.  We clearly have $R_pu + R_p\delta \subseteq I_{u, v(\delta)}$.  Lemma~\ref{lem:RuRdelta} gives the opposite containment, so we have equality.  Another application of Lemma~\ref{lem:RuRdelta} shows that $w\in\RO(I_{u,v(\delta)})$.
		
		Henceforth we assume that $c_p > 0$.  Using Lemma~\ref{lem:scalar-containment} and Lemma~\ref{lem:RuRdelta}, one can show that $R_pu + R_p\delta = p^{\bvchange{c_p}}I_{u', \bvchange{r_p}}$, where $u' := up^{-c_p}$.  Therefore, $w \in \RO(p^{\bvchange{c_p}}I_{u', \bvchange{r_p}}) = \RO(I_{u',\bvchange{r_p}})$.  By assumption $w,r_p$ satisfy the hypotheses of Proposition~\ref{prop:rightorders}, so $I_{u', \bvchange{r_p}}$ is the unique ideal of norm $p^\bvchange{r_p}$ such that $w \in \RO(I_{u', \bvchange{r_p}})$, and moreover, for any ideal $I$ such that $w \in \RO(I)$ we have $I \subseteq I_{u', \bvchange{r_p}}$.  By Lemma~\ref{lem:scalar-containment}, we also know that for any ideal $I$ of norm $p^{v(\delta)}$ such that $I \subseteq I_{u', \bvchange{r_p}}$ we have $u$, $\delta \in I$.  Thus it suffices to count the number of ideals $I$ of norm $p^{v(\delta)}$ such that $w \in \RO(I)$.  This is equal to the number of primitive ideals of norm $p^j$ where $j$ is at most $v(\delta)$ and $j\equiv v(\delta)\pmod 2$.  Applying Proposition~\ref{prop:rightorders} completes the proof.

		Since we have already shown that 
		\[
			\#\{ I\subseteq R : 
			I \textup{ satisfies }~\eqref{eq:IdealConditions}\} 
			= \prod_{p|\delta, p\ne\ell}
			\#\{ I_p\subseteq R_p : 
				I_p \textup{ satisfies }~\eqref{eq:IdealConditions}\},
		\]
		this proves Theorem~\ref{thm:CountingIdeals}.\qed



\begin{bibdiv}
		\begin{biblist}

			\bib{WIN2}{article}{
				author = {Anderson, J.},
				author = {Balakrishnan, J. S.},
				author = {Lauter, K.},
				author = {Park, J.},
				author = {Viray, B.},
				title = {Comparing arithmetic intersection theory formulas},
               conference={
                  title={Women in Numbers 2: Research Directions in Number Theory},
                  },
               book={
                  series={Contemporary Mathematics},
                  volume={606},
                  publisher={Amer. Math. Soc., Providence, RI},
               },
               date={2013},
               pages={65--82},
			}

			\bib{BY}{article}{
			   author={Bruinier, J.},
			   author={Yang, T.},
			   title={CM-values of Hilbert modular functions},
			   journal={Invent. Math.},
			   volume={163},
			   date={2006},
			   number={2},
			   pages={229--288},
			   issn={0020-9910},
			}

			\bib{Deuring}{article}{
			   author={Deuring, M.},
			   title={Die Typen der Multiplikatorenringe elliptischer
			   Funktionenk\"orper},
			   language={German},
			   journal={Abh. Math. Sem. Hansischen Univ.},
			   volume={14},
			   date={1941},
			   pages={197--272},
			}

			\bib{GL}{article}{
	   			author={Goren, E. Z.},
	   			author={Lauter, K. E.},
	   			title={Class invariants for quartic CM fields},
	   			language={English, with English and French summaries},
	   			journal={Ann. Inst. Fourier (Grenoble)},
	   			volume={57},
	   			date={2007},
	   			number={2},
	   			pages={457--480},
	   			issn={0373-0956},
			}
			
			\bib{GL10}{article}{
				author={Goren, E. Z.},
	   			author={Lauter, K. E.},
	   			title={Genus 2 Curves with Complex Multiplication}, 
				journal={International Mathematics Research Notices}, 
				date={2011},
				pages={75 pp.} 
			}

			\bib{Gross-CanonicalLifts}{article}{
			   author={Gross, B.},
			   title={On canonical and quasicanonical liftings},
			   journal={Invent. Math.},
			   volume={84},
			   date={1986},
			   number={2},
			   pages={321--326},
			   issn={0020-9910},
			}

			\bib{GrossKeating}{article}{
			   author={Gross, B.},
			   author={Keating, K.},
			   title={On the intersection of modular correspondences},
			   journal={Invent. Math.},
			   volume={112},
			   date={1993},
			   pages={225--245},
			}

			\bib{GZ-SingularModuli}{article}{
			   author={Gross, B.},
			   author={Zagier, D.},
			   title={On singular moduli},
			   journal={J. Reine Angew. Math.},
			   volume={355},
			   date={1985},
			   pages={191--220},
			   issn={0075-4102},
			}

            \bib{WIN}{article}{
               author={Grundman, H.},
               author={Johnson-Leung, J.},
               author={Lauter, K.},
               author={Salerno, A.},
               author={Viray, B.},
               author={Wittenborn, E.},
               title={Igusa class polynomials, embeddings of quartic CM fields, and
               arithmetic intersection theory},
               conference={
                  title={WIN---women in numbers},
               },
               book={
                  series={Fields Inst. Commun.},
                  volume={60},
                  publisher={Amer. Math. Soc., Providence, RI},
               },
               date={2011},
               pages={35--60},
            }

			\bib{HY}{book}{
			   author={Howard, B.},
			   author={Yang, T.},
			   title={Intersections of Hirzebruch-Zagier divisors and CM cycles},
			   series={Lecture Notes in Mathematics},
			   volume={2041},
			   publisher={Springer},
			   place={Heidelberg},
			   date={2012},
			   pages={viii+140},
			   isbn={978-3-642-23978-6},
			}

			\bib{Lang-Elliptic}{book}{
			   author={Lang, S.},
			   title={Elliptic functions},
			   series={Graduate Texts in Mathematics},
			   volume={112},
			   edition={2},
			   note={With an appendix by J. Tate},
			   publisher={Springer-Verlag},
			   place={New York},
			   date={1987},
			   pages={xii+326},
			}

			\bib{Lauter-Genus2Conj}{misc}{
				author={Lauter, K.},
				title={Primes in the denominators of Igusa Class Polynomials},
				note={Preprint, \texttt{arXiv:0301.240}},
			}

			\bib{LV-SingularModuli}{misc}{
				author = {Lauter, K.},
				author = {Viray, B.},
				title = {On singular moduli for arbitrary 
					discriminants},
				note = {Preprint, \texttt{arXiv:1206.6942}},
			}

			\bib{Vigneras}{book}{
			   author={Vign{\'e}ras, M.-F.},
			   title={Arithm\'etique des alg\`ebres de quaternions},
			   language={French},
			   series={Lecture Notes in Mathematics},
			   volume={800},
			   publisher={Springer},
			   place={Berlin},
			   date={1980},
			   pages={vii+169},
			   isbn={3-540-09983-2},
			}
			
			\bib{Waterhouse}{article}{
			   author={Waterhouse, W.},
			   title={Abelian varieties over finite fields},
			   journal={Ann. Sci. \'Ecole Norm. Sup. (4)},
			   volume={2},
			   date={1969},
			   pages={521--560},
			   issn={0012-9593},
			}

			\bib{Yang1}{article}{
				author={Yang, T.},
				title={An arithmetic intersection formula on Hilbert modular 
					surfaces},
				journal={Amer. J. Math.},
				volume={132},
				date={2010}, 
				pages={1275--1309},
			}

            \bib{Yang2}{article}{
               author={Yang, T.},
               title={Arithmetic intersection on a Hilbert modular surface and the
               Faltings height},
               journal={Asian J. Math.},
               volume={17},
               date={2013},
               number={2},
               pages={335--381},
               issn={1093-6106},
            }

		\end{biblist}
	\end{bibdiv}
\end{document}